\documentclass[11pt,oneside,english,letterpaper]{amsart}
\usepackage[utf8]{inputenc}
\usepackage{babel}
\usepackage{enumitem}
\usepackage{amstext}
\usepackage{amsthm}
\usepackage{amssymb}
\usepackage{setspace}
\usepackage{csquotes}
\usepackage[pdfusetitle,
 bookmarks=true,bookmarksnumbered=false,bookmarksopen=true,bookmarksopenlevel=1,
 breaklinks=false,pdfborder={0 0 1},backref=false,colorlinks=false]
 {hyperref}
\makeatletter
\numberwithin{equation}{section}
\numberwithin{figure}{section}
\theoremstyle{plain}
\newtheorem{thm}{\protect\theoremname}[section]
\theoremstyle{definition}
\newtheorem{defn}[thm]{\protect\definitionname}
\theoremstyle{plain}
\newtheorem{lem}[thm]{\protect\lemmaname}
\theoremstyle{remark}
\newtheorem{rem}[thm]{\protect\remarkname}
\theoremstyle{definition}
\newtheorem{example}[thm]{\protect\examplename}
\theoremstyle{plain}
\newtheorem{prop}[thm]{\protect\propositionname}
\theoremstyle{plain}
\newtheorem{cor}[thm]{\protect\corollaryname}

\usepackage{slashed}
\usepackage{xcolor}
\usepackage[all]{xy}

\def\pdv{\partial}
\newcommand{\gar}{\ifmmode{\mathrm{G\mathring{a}rding}}\else\text{Gårding}\fi}
\def\<{\langle}
\def\>{\rangle}

\def\R{\mathbb R}
\def\C{\mathbb C}

\def\N{\mathbb N}
\def\T{\mathbb T}

\def\a{\mathbf{a}}
\def\b{\mathbf b}
\def\c{\mathbf c}

\def\rr{\mathbf r}
\def\x{\mathbf{x}}
\def\y{\mathbf{y}}
\def\u{\mathbf{u}}
\def\v{\mathbf{v}}
\def\w{\mathbf{w}}
\def\z{\mathbf{z}}
\def\e{\mathbf{e}}
\def\L{\mathrm L} %lorenztian
\def\P{\mathrm P}
\def\M{\mathrm M}

\def\T{\mathrm T}

\def\S{\mathrm S} %stable polynomials

\def\G{{\mathrm G}^\mathfrak{a}}
\def\Gh{{\mathrm G}^\mathfrak{h}}
\def\Gah{{\mathrm G}^\mathfrak{ah}}
\def\GS{\mathrm G}
\def\I{\mathrm I}
\def\Ia{{\mathrm I}^\mathfrak{a}}

\def\U{\mathrm U}
\def\C{\mathcal C}

\def\A{\mathcal{A}}
\def\PS{\Xi}
\newcommand{\MRS}{\mathcal{M}}
\newcommand{\rhat}{\widehat{\rr}}

\renewcommand{\rm}{r_{-1}}
\newcommand{\rminus}{\rm}

\newcommand{\upoly}{\mathfrak p}
\newcommand{\uquot}{\mathfrak q}

\newcommand{\pdvminus}{\partial_{-}}

\newcommand{\Ad}{A_d}

\newcommand{\gapmap}{A}

\def\proj{\Pi^\downarrow} % projection/restriction

\def\pol{\Pi^\uparrow} % polarization

\def\ppol{\hat{\Pi}} %partial polarization

\def\Sym{\mathrm {sym}}

\def\HH{\mathrm H} %homogenization

\def\Vol{\mathrm{Vol}}%volume

\def\branden{Br\"and\'en}
\def\vtl{\vartriangleleft}
\def\bvtl{\blacktriangleleft}
\newcommand{\ip}{\mathrel{\prec\!\prec}} %ideal position

\providecommand{\corollaryname}{Corollary}
\providecommand{\definitionname}{Definition}
\providecommand{\examplename}{Example}

\providecommand{\lemmaname}{Lemma}
\providecommand{\propositionname}{Proposition}
\providecommand{\remarkname}{Remark}
\providecommand{\theoremname}{Theorem}

\AddToHook{package/after/biblatex}{
\AtEveryBibitem{\clearfield{note}}}

\makeatother
\usepackage[style=numeric,sorting=nyt,doi=false,url=false,maxbibnames=10]{biblatex}
\providecommand{\corollaryname}{Corollary}
\providecommand{\definitionname}{Definition}
\providecommand{\examplename}{Example}
\providecommand{\lemmaname}{Lemma}
\providecommand{\propositionname}{Proposition}
\providecommand{\remarkname}{Remark}
\providecommand{\theoremname}{Theorem}

\begin{document}
\title{Ideal Gårding polynomials}
\author{Hao Fang}
\email{hao-fang@uiowa.edu}
\address{Department of Mathematics, University of Iowa, Iowa City, IA 52246,
USA}
\thanks{H. F.'s work is partially supported by a Simons Foundation mathematics
collaboration grant and an NSF-RTG grant. B. M. is partially supported
by NSFC grant 12471052.}
\author{Biao Ma}
\email{bma@math.ecnu.edu.cn}
\address{School of Mathematical Sciences, East China Normal University, 500
Dongchuan road, Shanghai, China.}

\begin{abstract} We introduce ideal G\aa{}rding polynomials, a convexity-enhanced subclass of G\aa{}rding polynomials whose G\aa{}rding components are recursively convex under partial differentiation. This class strictly contains real stable polynomials and, after translation and homogenization, lies in the Lorentzian class. 

Our main result is that ideal G\aa{}rding polynomials still admit a robust structure theory despite this additional convexity: they are preserved under polarization, satisfy natural closure properties, and support a linear preserver theory. A key contribution of this paper is a universal model for univariate G\aa{}rding polynomials, described by monotone root sequences and equivalently by volume polynomials of Pitman--Stanley polytopes. We establish quotient concavity, and Newton--Maclaurin type inequalities, which leads to the polarization theorem, and suggests further connections with convex geometry and Lorentzian polynomials.
\end{abstract}
\maketitle
\tableofcontents
\section{Introduction}
In this paper, we introduce \emph{ideal G{\aa}rding polynomials}, an intermediate class between stable and G{\aa}rding polynomials that restores convexity within the G{\aa}rding framework.

In earlier work~\cite{Fang-MaGard}, we introduced G{\aa}rding polynomials from the positive affine geometry of their positivity regions. This class properly extends the class of real stable polynomials and shares several of its structural features.

In particular, G{\aa}rding polynomials encode ellipticity and satisfy the Rayleigh property, with applications in combinatorics and probability. In the multi-affine case, they are characterized by the existence of a unique component passing the positive ray test; in general, they are defined by invariance under strictly positive affine maps.

However, this construction has a fundamental limitation: although the G{\aa}rding polynomials encode ellipticity through positivity regions, they do not guarantee convexity. Indeed, the G\aa{}rding component, which is the analog of the G{\aa}rding cone for a hyperbolic polynomial, does not need to be convex.

This contrasts with the classical hyperbolic polynomial theory, where G{\aa}rding~\cite{Gar59} showed that a hyperbolic polynomial is associated with a positivity region, called the G\aa{}rding cone, which is a convex cone. Similar convexity phenomena also extend to real stable polynomials; see \cite{BB09I,BBL09,Pem11}. More recently, the theory of Lorentzian polynomials introduced by \branden-Huh and Anari-Liu-OveisGharan-Vinzant ~\cite{BrandenHuh20,ALOVi,ALOVii} provides a unified framework
capturing strong log-concavity phenomena in the positive hyper-octant.  These convexity properties are closely related to concavity inequalities and play a central role in fully nonlinear PDE, optimization, combinatorics and probability \cite{caffarelli1985dirichlet,BB08Duke,MR2661242,BB09Annals,BB09I,B07Adv,BBL09,Wagner11,KPV15,Guler97,MR3055586,MSS15i,MSS15ii,MR2411443,BrandenHuh20,ALOVi,ALOVii,ALOViii}.
  
To address this issue, we introduce a more restrictive class.

\begin{defn}[Ideal \gar{} Polynomials]\label{def:ideal}
Denote $\GS_n^d$ the set of  $n$-variate \gar{} polynomials of degree $d$. We define the ideal G{\aa}rding polynomials recursively as follows. Set
\[
\I_n^0:=\{f\equiv c:\ c\ge 0\},\qquad
\I_n^1:=\GS_n^1.
\]
For $d\ge 2$, let
\[
\I_n^d
:=
\bigl\{
f : \partial_i f\in \I_n^{d-1}\ \forall i\in[n],\ 
\emptyset\neq \C_f \subset \bigcap_{i=1}^n \C_{\partial_i f},\ 
\C_f \text{ is convex}
\bigr\}.
\] Here $\C_f$ denote the \gar{} component of $f$. See Section \ref{subsec:garcompoents} for details.
\end{defn}

Thus, ideal G{\aa}rding polynomials are G{\aa}rding polynomials whose distinguished components are convex and stable under differentiation.

\begin{rem}
For multi-affine polynomials, $f$ is ideal G{\aa}rding if and only if $f$ is G{\aa}rding and its G{\aa}rding component $\C_f$ is nonempty and convex.

Indeed, in the multi-affine case, $f$ is G{\aa}rding if and only if $\C_f\neq\emptyset$. Moreover, whenever $\partial_i f \not\equiv 0$, the convexity of $\C_{\partial_i f}$ can be deduced from the convexity of $\C_f$ by \cite{Fang-MaGard} Lemma 5.2. Hence, the recursive condition in Definition~\ref{def:ideal} reduces to the convexity of $\C_f$.
\end{rem}

Our main result is the following structure theorem.
%\ref{enu:mainthmideal} $\Leftrightarrow$ \ref{enu:maththmpullback} $\Leftrightarrow$ \ref{enu:maththmpolar}
\begin{thm}\label{thm:mainstructure}
Let $f$ be a degree $d$ G{\aa}rding polynomial. The following are equivalent:
\begin{enumerate}[label=(\arabic*)]
\item $f$ is ideal G{\aa}rding; \label{enu:mainthmideal}
\item $f$ is the pullback of a multi-affine ideal G{\aa}rding polynomial by a strictly positive affine map;\label{enu:maththmpullback}
\item a $\kappa$-polarization of $f$ is a multi-affine ideal G{\aa}rding polynomial;\label{enu:maththmpolar}
\item \label{enu:maththmIdealtoconcave}for every nontrivial $\partial^{\alpha}f$, $(\partial^{\alpha}f)^{1/\deg(\partial^{\alpha}f)}$ is concave on $\C_{\partial^{\alpha}f}$;
\item \label{enu:maththmIdeallogconcave}for every nontrivial $\partial^{\alpha}f$, $\log \partial^{\alpha}f$ is concave on $\C_{\partial^{\alpha}f}$;
\item for every nontrivial $\partial^{\alpha}f$, $\x_0\in \C_{\pdv^\alpha f}$, the localized homogenization
\[
\HH_{\x_0} \pdv^\alpha f := \pdv^\alpha f\bigl(\x_0+\x/y\bigr)y^{d-|\alpha|}
\]
is Lorentzian.\label{enu:maththmIdealtLorentzian}
\end{enumerate}
\end{thm}

Theorem~\ref{thm:mainstructure} extends Theorem 1.1 of~\cite{Fang-MaGard}, a structural theorem for G\aa{}rding polynomials. It provides equivalent geometric, algebraic, and analytic characterizations of ideal G{\aa}rding polynomials. In particular, condition~(6) shows that the localized homogenizations are Lorentzian.

We also identify natural examples.
 
\begin{thm}\label{thm:basicfamilies}
Let $f$ be a G{\aa}rding polynomial. Then $f$ is ideal G{\aa}rding if
\begin{enumerate}[label=(\arabic*)]
\item $f$ is real stable with positive leading coefficients;
\item $f$ is a polarization of a univariate G{\aa}rding polynomial;\label{part2}
\item $f$ is the eigen-polynomial of a nonnegative matrix.
\end{enumerate}
\end{thm}

\ref{part2} of Theorem~\ref{thm:basicfamilies} is a special case of the $\kappa$-polarization result in Theorem~\ref{thm:mainstructure}; a special case of \ref{part2}, where the number of variables equals the degree, was first established by Lin~\cite{Lin23JFA,LIN2024}. Our approach is different from that of Lin.

We also establish a symbol based theory of linear preservers for ideal G{\aa}rding polynomials. Such theories are rare and typically reflect deep structural properties of the underlying polynomial class. In addition to the linear preserver theory for real stable polynomials developed by Borcea and Brändén \cite{BB09I} and the corresponding theory for Lorentzian polynomials \cite{BrandenHuh20}, our previous work \cite{Fang-MaGard} showed that Gårding polynomials also admit a similar theory. From this point of view, ideal Gårding polynomials form a natural intermediate class between stable and Gårding polynomials.

One main contribution of our paper is a geometric model for univariate G{\aa}rding polynomials via Pitman--Stanley polytopes \cite{PS}, providing a volume interpretation. This yields quotient concavity and Newton--Maclaurin type inequalities and leads to the polarization theorem.

From the PDE perspective, ideal G{\aa}rding polynomials provide
a flexible source of fully nonlinear elliptic operators of Hessian type. They are well suited for the
Caffarelli--Nirenberg--Spruck theory \cite{caffarelli1985dirichlet}, where both ellipticity and convexity/concavity are essential. These examples
include the Monge--Amp\`ere equation, $\sigma_k$-Hessian equations, the special Lagrangian equation, the $J$-equation, inverse $\sigma_k$ equations, and many more. The related works are extensive; see, for instance
\cite{CGY1chang2002equation,HarveyLawson82,yuan06,chang2010liouville,hou2010second,chou-wang2001,CHY2chang2011prescribing,GSgursky2018formal,MR1719551,viaclovskytrans2000,shankar2021rigidity,collins2017convergence,chen2021j,Fang-Lai-Ma,MR1284912,MR2487853,szekelyhidi2018fully,Song2020NakaiMoishezonCF,datar2021numerical,MR4278951,Lin23JFA,lin2023solvabilitygeneralinversesigmak,FANG2024109867}. 

In a recent preprint~\cite{FangMaWuLiouville}, together with J. Wu, we establish a Liouville-type theorem for degenerate elliptic Hessian equations, for which ideal G{\aa}rding polynomials provide a natural source of examples. Further connections with K\"ahler geometry as in ~\cite{FANG2024109867} will be explored in future work.

The paper is organized as follows. In Section~2, we recall basic properties of G{\aa}rding polynomials. In Section~3, we establish elementary properties of the ideal G{\aa}rding polynomials. Sections~4–6 develop the universal quotient theory and its geometric interpretation. In Section~7, we prove the polarization theorem. The remaining sections develop the linear preserver theory, closure properties, and applications to concavity and Lorentzian polynomials.

\subsection*{Acknowledgments}
The authors thank Gao Chen for his interest in this work and for pointing out an error in an earlier version of the manuscript.
\section{Preliminaries}

We fix some notation and briefly recall results on \(\gar{}\) polynomials from our previous work~\cite{Fang-MaGard}. We also review the polynomial transformations that will be used in later sections.

\subsection{Notation}

Let \([n]=\{1,2,\ldots,n\}\), and let \(\R[\x]\) denote the polynomial ring in the variables \(\x=(x_1,\ldots,x_n)\). Set
\(
\mathbb N:=\{0,1,2,\ldots\}.
\)
For a multi-index \(\alpha=(\alpha_1,\ldots,\alpha_n)\in\mathbb N^n\), write
\[
|\alpha|:=\sum_{i=1}^n \alpha_i,\qquad
\partial^\alpha:=\partial_1^{\alpha_1}\cdots \partial_n^{\alpha_n},\qquad
\x^\alpha:=x_1^{\alpha_1}\cdots x_n^{\alpha_n},
\]
where 
\(
\partial_i:=\frac{\partial}{\partial x_i}.
\)      

For \(k\in\mathbb N\), let
\[
\sigma_k(\x):=
\sum_{1\le i_1<\cdots<i_k\le n}\prod_{j=1}^k x_{i_j}
\]
be the \(k\)-th elementary symmetric polynomial. We use the conventions
\[
\sigma_0(\x)=1,\qquad \sigma_k(\x)=0\ \text{for }k>n.
\]
For a degree $d$ polynomial $f$, its homogenization at a point $\x_0$ is defined as
\[
\mathrm{H}_{\x_0}f(\mathbf{x},y):=f\left(\x_0+\frac{\mathbf{x}}{y}\right)y^{d}.\]
If $\x_0=\mathbf{0}$ then we simply denote $\HH f(\x,y)=\HH_\mathbf{0}f(\x,y)$.

For \(\kappa=(\kappa_1,\ldots,\kappa_n)\in\mathbb N^n\), define
\[
\R_\kappa[\x]
:=
\{\,f\in\R[\x]: \deg_{x_i}f\le \kappa_i \text{ for all } i\in[n]\,\}.
\]   

A polynomial \(f(\x)\) is \emph{multi-affine} if \(\deg_{x_i}f\le 1\) for every \(i\in[n]\). We use the superscript \(\mathfrak a\) for multi-affine polynomials, the superscript \(\mathfrak h\) for homogeneous polynomials, and the subscript \(+\) for polynomials with nonnegative coefficients. For example,
\[
\R^{\mathfrak a}_{+}[\x]
:=
\{\,f\in\R[\x]: \deg_{x_i}f\le 1\ \text{for all }i,\ \text{and }f\text{ has nonnegative coefficients}\,\}.
\]

\subsection{Positive affine maps and \gar{} components} \label{subsec:garcompoents}
Fix the standard coordinate frame on \(\R^n\). For \(\x\in\R^n\), write \(\x>0\) (resp.\ \(\x\ge 0\)) if \(x_i>0\) (resp.\ \(x_i\ge 0\)) for all \(i\). The set
\[
\Gamma_n^+:=\{\x\in\R^n:\x>0\}
\]
is the positive hyper-octant. A \emph{positive affine hyper-octant} is a translate \(\x_0+\Gamma_n^+\) with \(\x_0\in\R^n\).

An affine map \(\mu:\R^n\to\R^m\) has the form
\[
\mu(\x)=A\x+\y_0,
\]
where \(A=(a_{ij})\in M_{m\times n}(\R)\) and \(\y_0\in\R^m\). We call \(\mu\)
\begin{itemize}
\item \emph{positive} (or \emph{weakly positive}) if \(a_{ij}\ge 0\) for all \(i,j\);
\item \emph{strictly positive} if \(\mu\) is positive and each row sum of \(A\) is strictly positive, i.e.\,
\[
\sum_{j=1}^n a_{ij}>0\qquad\text{for every }i\in[m].
\]
\end{itemize}
Let \(\A_+\) and \(\A_{++}\) denote the sets of positive and strictly positive affine maps, respectively.

Let \(f(\x)\in\R[\x]\). Define the \(\gar{}\) component \(\C_f\subset\R^n\) to be the distinguished connected component of
\(
\{\x\in\R^n:f(\x)>0\}
\)
that passes the \emph{positive ray test}, namely
\begin{equation}\label{eq:PRT}
\C+\Gamma_n^+\subset \C.
\end{equation}

In~\cite{Fang-MaGard}, the class of G{\aa}rding polynomials was introduced and characterized in several equivalent ways. We use the following definition.

\begin{defn}[\gar{} polynomials]
\label{def:garding}\label{def:general garding}
Let
\[
\GS_n^0:=\{f\equiv c:c\ge 0\},
\qquad
\GS_n^1:=\{f(\x):\deg f=1,\,\C_f\neq\emptyset\}.
\]
For \(d\ge 2\), define
\[
\GS_n^d
:=
\left\{
f(\x):
\partial_i f\in \GS_n^{d-1}\ \text{for all }i\in[n],\
\C_f\neq\emptyset,\
\C_f\subset \bigcap_{i=1}^n \C_{\partial_i f}
\right\}.
\]
We write
\[
\GS:=\bigcup_{n,d}\GS_n^d,\qquad
\GS_n:=\bigcup_d \GS_n^d,\qquad
\GS^d:=\bigcup_n \GS_n^d.
\]
We also denote by \(\G\) the set of multi-affine \(\gar{}\) polynomials, by \(\Gh\) the set of homogeneous \(\gar{}\) polynomials, and by \(\Gah:=\G\cap\Gh\).
\end{defn}

\subsection{Polynomial transformations}

The following operations will be used repeatedly.

\begin{defn}[Polynomial transformations]
\label{def:maps}
Let \(f\in\R[\x]\).
\begin{enumerate}[label=(\arabic*)]
\item \label{enu:pullbck}\emph{Pullback by positive affine maps.}
If \(\mu:\R^n\to\R^m\) is a positive affine map, define \(\mu^*:\R[\y]\to\R[\x]\) by
\[
(\mu^*f)(\x):=f(\mu(\x)).
\]
Thus both \(\A_+\) and \(\A_{++}\) act on polynomials by pullback.

\item \label{enu:restr}\emph{Restriction (specialization).}
For \(i\in[n]\) and \(a\in\R\), define
\[
f|_{x_i=a}:=f(x_1,\ldots,x_{i-1},a,x_{i+1},\ldots,x_n).
\]

\item \label{enu:polytransprojection}\emph{Projection (diagonal specialization).}
The diagonal restriction of \(f\) is the univariate polynomial
\[
\proj(f):=f(t,\ldots,t).
\]

\item \emph{Polarization.}
If \(d\le n\), let
\(
\pol:\R_d[x]\to \R^{\mathfrak a}[\x]
\)
be the unique linear map satisfying
\[
\pol(x^k)=\frac{\sigma_k(\x)}{\binom{n}{k}}.
\]

\item \emph{\(\kappa\)-projection.}
For \(\kappa\in\N^n\), let \(\x_\kappa:=(x_{ij})_{i\in[n],\,j\in[\kappa_i]}\). If \(f\in\R_\kappa[\x_\kappa]\), define
\[
\proj_\kappa f:=f|_{x_{ij}=x_i},
\qquad i\in[n],\ j\in[\kappa_i].
\]

\item \label{enu:polytranspolar}\emph{\(\kappa\)-polarization.}
For \(\kappa\in\N^n\) and \(f\in\R_\kappa[\x]\), let
\[
\pol_\kappa:\R_\kappa[\x]\to \R^{\mathfrak a}[\x_\kappa]
\]
be the linear map defined on monomials by
\[
\pol_\kappa(\x^\alpha)
:=
\prod_{i=1}^n
\frac{\sigma_{\alpha_i}(x_{i1},\ldots,x_{i\kappa_i})}{\tbinom{\kappa_i}{\alpha_i}}.
\]

\item \label{enu:polytransPartial}\emph{Partial and directional differentiation.}
For \(i\in[n]\), set
\[
\partial_i f(\x):=\frac{\partial f}{\partial x_i}(\x).
\]
More generally, for \(\a\in\R^n\) with \(\a\ge 0\), define
\[
D_\a f(\x):=\sum_{i=1}^n a_i\,\partial_i f(\x).
\]
\end{enumerate}
\end{defn}

We will use the following preservation theorem from~\cite{Fang-MaGard}.

\begin{thm}
\label{thm:maps-preserve-garding}
Let \(T\) be a transformation in Definition~\ref{def:maps}. If \(f(\x)\in \GS_+[\x]\), then \(Tf\) is also a \(\gar{}\) polynomial. Moreover,  for every \(f\in\GS[\x]\), if \(T\) is one of the transformations in \ref{enu:polytransprojection}--\ref{enu:polytransPartial} or $T=\mu^*$ for some strictly positive affine map $\mu\in\A_{++}$, then \(Tf\in \GS\). 
\end{thm}

For the proof, see~\cite[Section~8]{Fang-MaGard} for the transformations \ref{enu:pullbck}--\ref{enu:polytranspolar}, and~\cite[Section~10]{Fang-MaGard} for \ref{enu:polytransPartial}.

The corresponding preservation results for ideal G{\aa}rding polynomials
will be proved later in Section~\ref{thm:positive-affine-pullback-and-directional-derivatives}.

\begin{cor}
\label{cor:basicrestriction}
Suppose \(\x\in\R^n\) and \(\y\in\R^m\), and let
\[
\pi:\R^{n+m}\to\R^n,\qquad (\x,\y)\mapsto \x.
\]
Let \(f(\x,\y)\in \GS_{\kappa\oplus\gamma}[\x,\y]\). Suppose that for some \(\alpha\in\N^m\),
\[
\partial_\y^\alpha f(\x,\y)
:=
\Bigl(\prod_{i=1}^m \partial_{y_i}^{\alpha_i}\Bigr)f(\x,\y)\not\equiv 0.
\]
Then, for every \(\x_0\in \pi(\C_{\partial_\y^\alpha f})\), the specialization
\(
\phi_{\x_0}(\y):=f(\x_0,\y)
\)
belongs to \(\GS[\y]\).
\end{cor}

\begin{proof}
By \cite[Section~8]{Fang-MaGard}, \(\kappa\)-polarization and \(\kappa\)-projection preserve \(\gar{}\) polynomials. After polarizing in the \(\y\)-variables we may assume that \(f\) is multi-affine in \(\y\), that is, \(\deg_{y_i}f\le 1\) for all \(i\in[m]\). The claim then follows from~\cite[Lemma~5.15]{Fang-MaGard}.
\end{proof}

\subsection{Rayleigh property}

A polynomial \(f(\x)\) with nonnegative coefficients is called \emph{Rayleigh} if for every \(\x\ge 0\), every \(\alpha\in\N^n\), and all \(i,j\in[n]\),
\begin{equation}\label{eq:-18}
\partial^\alpha f(\x)\,\partial^{\alpha+\e_i+\e_j}f(\x)
-
\partial^{\alpha+\e_i}f(\x)\,\partial^{\alpha+\e_j}f(\x)
\le 0.
\end{equation}

\begin{thm}
\label{thm:GardtoRayleigh}
If \(f(\x)\in \GS_+[\x]\), then \(f\) is Rayleigh.
\end{thm}

\begin{proof}
It suffices to consider \(\alpha=\mathbf 0\), since each \(\partial^\alpha f\) is again a \(\gar{}\) polynomial.

If \(f\) is multi-affine, the claim is~\cite[Lemma~5.12]{Fang-MaGard}. Otherwise, choose \(\kappa\in\N^n\) such that \(f\in \GS_{\kappa,+}[\x]\), with \(\kappa_i\ge 2\) whenever needed, and let
\[
\x_\kappa=(x_{ij})_{i\in[n],\,j\in[\kappa_i]},
\qquad
F(\x_\kappa):=\pol_\kappa f(\x_\kappa).
\]
Then for \(\x_\kappa\ge 0\),
\[
F\,\partial_{x_{i1}}\partial_{x_{j2}}F
-
\partial_{x_{i1}}F\,\partial_{x_{j2}}F
\le 0.
\]
Using the polarization identity
\[
\kappa_i\,\partial_{x_{i1}}\pol_\kappa(f)
=
\pol_{\kappa-\e_i}(\partial_i f)
\]
and then applying \(\kappa\)-projection, we obtain
\[
f(\x)\,\partial_i\partial_j f(\x)-c\,\partial_i f(\x)\,\partial_j f(\x)\le 0,
\]
where \(c=1-\frac{1}{\kappa_i}\) if \(i=j\), and \(c=1\) if \(i\ne j\). Hence \(f\) is Rayleigh.
\end{proof}

\subsection{An auxiliary lemma}
We record a simple lemma that will be used repeatedly.

\begin{lem}
\label{lem:bootstrapconcavity}
Let \(U\subset\R^n\) be open and convex, and let \(g,h:U\to(0,\infty)\). If for some \(d>0\), the functions \(g^{1/d}\) and \(h\) are concave on \(U\), then \((gh)^{1/(d+1)}\) is also concave on \(U\).
\end{lem}

\begin{proof}
Set \(f:=gh\). For \(t\in[0,1]\) and \(\a,\b\in U\),
\[
g(t\a+(1-t)\b)^{1/d}\ge t\,g(\a)^{1/d}+(1-t)\,g(\b)^{1/d},
\]
and
\[
h(t\a+(1-t)\b)\ge t\,h(\a)+(1-t)\,h(\b).
\]
 
Set \(z:=t\a+(1-t)\b\). Then
\[
\begin{aligned}
f(z)^{\frac1{d+1}}
&= g(z)^{\frac{d}{d+1}}\,h(z)^{\frac1{d+1}} \\
&\ge
\bigl(t\,g(\a)^{\frac1{d}}+(1-t)\,g(\b)^{\frac1{d}}\bigr)^{\frac{d}{d+1}}
\bigl(t\,h(\a)+(1-t)\,h(\b)\bigr)^{\frac1{d+1}} \\
&\ge
t\,(g(\a)h(\a))^{\frac1{d+1}}
+(1-t)\,(g(\b)h(\b))^{\frac1{d+1}}.
\end{aligned}
\]
where the last line is Hölder's inequality. Thus, \((gh)^{1/(d+1)}\) is concave on \(U\).
\end{proof}
\section{Elementary Properties}

We record some immediate consequences of the definition of
ideal G{\aa}rding polynomials. In particular, we establish basic stability
properties under translation and topological limits.
 \begin{lem} 
\label{lem:translation-invariance}
Let $f\in \I_n$ and let $\a\in\R^n$. Define the translation
\[
\T_\a(\x):=\x+\a,
\qquad
\T_\a^* f(\x):=f(\T_\a(\x))=f(\x+\a).
\]
Then
\(
\T_\a^* f \in \I_n.
\)
Moreover, if $\a\in \C_f$, then
\(
\T_\a^* f \in \I_{+,n}.
\)
\end{lem}

\begin{lem} \label{lem:restrictionI_+} Let $f(\x,\y)\in \I_+[\x,\y]$. Then $f(\x,\mathbf{0})\in \I_+[\x]$. 
\end{lem}
\begin{proof}
    Let $g(\x):=f(\x,\mathbf{0})$. Since the restriction operation preserves $\GS_+$ by Theorem \ref{thm:maps-preserve-garding},  $f(\x,\mathbf{0})\in\GS_+[\x]$. Then \[\C_{\pdv^\alpha g}=\{\x:\,(\x,\mathbf{0})\in \C_{\pdv^\alpha f}\}.\] Since $f$ is ideal, $\C_{\pdv^\alpha f}$ is convex and $\C_{\pdv^\alpha g}$ as a slice of a convex set is also convex. Hence $g(\x)\in \I_+[\x]$.
\end{proof}

\begin{lem} 
\label{lem:external-products}
Let \(f\in\I[\x]\) and \(g\in\I[\y]\), where the variables \(\x\) and
\(\y\) are disjoint. Then
\[
f(\x)g(\y)\in\I[\x,\y].
\]
The same statement holds for \(\I_+\).
\end{lem}
\begin{proof}
By definition, for disjoint variables,
$\C_{fg}=\C_f\times\C_g$, which is convex. Moreover,
$\partial_\x^\alpha\partial_\y^\beta(fg)
=(\partial_\x^\alpha f)(\partial_\y^\beta g)$,
so the recursive conditions are inherited from those of $f$ and $g$.
Hence $f(\x)g(\y)\in\I[\x,\y]$. The $\I_+$ case is immediate.
\end{proof}
\begin{thm}
\label{thm:I-closure-properties}
The following hold.

\begin{enumerate}
\item[\textup{(a)}]
The class \(\I_+[\x]\) is closed in \(\R[\x]\) with respect to the compact--open topology.

\item[\textup{(b)}]
Suppose \(f\in\I_+\) and
\[
f=f_d+f_{d-1}+\cdots+f_m,
\]
where each \(f_k\) is homogeneous of degree \(k\), \(f_d\neq0\), and
\(f_m\neq0\). Then
\[
f_d\in\I_+^{\mathfrak h},
\qquad
f_m\in\I_+^{\mathfrak h}.
\]
\end{enumerate}
\end{thm}

\begin{proof}
We prove \textup{(a)} directly from the recursive definition of \(\I\).

Let \(f_r\in\I_+[\x]\) converge to \(f\in\R[\x]\) in the compact--open
topology. Then \(f\) has nonnegative coefficients. For every multi-index
\(\alpha\), we have \(\partial^\alpha f_r\to\partial^\alpha f\) in the
compact--open topology. Since \(f_r\in\I_+\subset\GS_+\), the closedness
theorem for G{\aa}rding polynomials, \cite[Lemma~5.18]{Fang-MaGard},
implies that every nonzero derivative \(\partial^\alpha f\) is
G{\aa}rding.

It remains to check convexity of $\C_{\pdv^\alpha f}$. For each
\(\alpha\), the components \(\C_{\partial^\alpha f_r}\) are convex. By the
same convergence result for G{\aa}rding components, these components converge
locally to \(\C_{\partial^\alpha f}\). A local Hausdorff limit of convex
sets is convex. Hence \(\C_{\partial^\alpha f}\) is convex for every
nonzero \(\partial^\alpha f\). The inclusions
\(\C_{\partial^\alpha f_r}\subseteq\bigcap_i\C_{\partial_i\partial^\alpha f_r}\)
also pass to the limit. Therefore every nonzero derivative of \(f\)
satisfies the recursive ideal G{\aa}rding condition, and hence
\(f\in\I_+[\x]\). This proves \textup{(a)}.

For \textup{(b)}, first note that positive diagonal rescaling preserves
\(\I_+\) directly from the definition. Thus, for \(t>0\),
\[
t^{-d}f(t\x)\in\I_+.
\]
As \(t\to\infty\), \(t^{-d}f(t\x)\to f_d\) in the compact--open topology.
By \textup{(a)}, \(f_d\in\I_+\). Since \(f_d\) is homogeneous,
\(f_d\in\I_+^{\mathfrak h}\). Similarly, \(t^{-m}f(t\x)\in\I_+\) for
\(t>0\), and as \(t\to0^+\), \(t^{-m}f(t\x)\to f_m\). Again by
\textup{(a)}, \(f_m\in\I_+^{\mathfrak h}\).
\end{proof}

\section{Stable and Matrix Examples}
We present two basic families of examples of ideal
\(\gar{}\) polynomials: positive real stable polynomials and
natural matrix-theoretic characteristic polynomials.

\subsection{Positive real stable polynomials}
Recall that in~\cite{Fang-MaGard} we denoted by \(\S[\x]\) the class of
real stable polynomials with positive leading coefficients and proved in
Theorem~4.16 that
\[
\S[\x]\subset \GS[\x].
\]
We strengthen this result by showing that the real stable
polynomials in $\S[\x]$ are ideal G{\aa}rding polynomials. The proof follows the original idea of \gar{}~\cite{Gar59}.  

We use the following
form of Bochner's tube theorem~\cite[V, Theorem~9]{BochnerMartin1948}.

\begin{thm}[Bochner's tube theorem]
\label{thm:Bochner}
Let $U\subset\R^n$ be a connected open set, and let
$\operatorname{conv}(U)$ denote its convex hull.
If $g(\mathbf z)$ is holomorphic on the tube domain
\(
U+i\R^n,
\)
then $g$ admits an analytic extension to
\(
\operatorname{conv}(U)+i\R^n.\)

\end{thm}
\begin{lem}
\label{lem:stable-tube-zero-free}
Let $f\in\S[\x]$, and let $\C_f$ be its positive G{\aa}rding component.
Then
\[
f(\z)\neq 0
\qquad
\text{for all } \z\in \C_f+i\R^n .
\]
\end{lem}

\begin{proof}
Let $\mathbf a\in \C_f$. Let $\HH f(\x,y)$ be the homogenization of $f$.  Then by~\cite[Lemma 4.3, Lemma 4.6]{Fang-MaGard}, $(\mathbf a,1)$ belongs to the corresponding G\aa rding cone of $\HH f$, and   $\HH f$ is hyperbolic with
respect to $(\mathbf a,1)$.  Therefore,
\[f(\mathbf a+i\mathbf y)=
\HH f\bigl((\mathbf a,1)+i(\mathbf y,0)\bigr)\neq 0
\qquad
\text{for all } \mathbf y\in\R^n .
\]

Thus $f$ has no zeros in $\C_f+i\R^n$.
\end{proof}

\begin{prop}
\label{prop:stable-polynomials-are-ideal}
 
\(
\S[\x]\subset\I[\x].
\)
\end{prop}

\begin{proof}
We first show that $\C_f$ is convex.  By
Lemma~\ref{lem:stable-tube-zero-free}, $f$ has no zeros in the tube
domain
\(
\Omega:=\C_f+i\R^n .
\)
Hence $1/f$ is holomorphic on $\Omega$. By Theorem~\ref{thm:Bochner}, $1/f$ admits a holomorphic extension to
\(\operatorname{conv}(\C_f)+i\R^n.\)
On $\Omega$, this extension agrees with $1/f$. Consequently,
\(f(\z)\neq 0\) in $\operatorname{conv}(\C_f)$. Because $\C_f$ is the connected component of $\{f>0\}\subset \R^n$, we have
\(\operatorname{conv}(\C_f)=\C_f,\)
and therefore $\C_f$ is convex. Applying the same argument to  $\partial^\alpha f$ shows that every
nontrivial derivative G{\aa}rding component of $f$ is convex.  Hence
\(
f\in\I[\x].
\)
\end{proof}

\subsection{Non-negative matrices and $M$-matrices}

A real matrix is called non-negative if all the entries are non-negative
real numbers. Non-negative matrices are fundamental subjects in the matrix theory
with numerous applications. See \cite{BP79Nonnegative}. We show that multi-variate eigen-polynomials for non-negative matrices
are in fact ideal \gar{}. 

\begin{thm}\label{thm:nonnegmatrix}For
any non-negative matrix $N$,
\[ p_N(\x):=\det(\mathrm{diag}(\x)-N)\in\I[\x].
\] 
\end{thm}
This theorem is a refinement of  Theorem 16.1 of \cite{Fang-MaGard}, which shows that $ p_N(\x)$ is \gar{}.

We introduce the notion of the $M$-matrix.
\begin{defn}[$M$-matrix]
An $n \times n$ matrix $A$ is called an $M$-matrix if
\[
A=sI_n-B
\]
for some entrywise nonnegative matrix $B=(b_{ij})\in \mathbb{R}^{n\times n}$ and some real number $s\ge \rho(B)$.
Here $\rho(B)$ denotes the spectral radius of $B$.
See \cite[Chapter 6]{BP79Nonnegative} for 50 equivalent characterizations of $M$-matrices.
\end{defn}

Theorem 16.1 in \cite{Fang-MaGard} characterizes the \gar{} component of $p_N$ as 
\[
\C_{p_{N}}=\{\x|\ \mathrm{diag}(\x)-N\text{ is a non-singular \ensuremath{M}-matrix}\}.
\]

To prove Theorem \ref{thm:nonnegmatrix}, we need the following lemma.
\begin{lem}
\label{lem:M-matrixCM}Let $A$ be an $M$-matrix and $Z=\mathrm{diag}(z_{1},\cdots,z_{n})$.
Then, $Z+A$ is invertible if $\Re z_{i}>0$ for all $i\in[n]$. 
\end{lem}

\begin{proof}
Given $A$, we decompose $A=sI_{n}-B$ where $s\geq\rho(B)$. If $\Re z_{i}>0$
for all $i\in[n]$, \[|z_i+s|^2=(\Re{z_i}+s)^2+|\Im z_i|^2>s^2>\rho^2(B).\] Take $C=(c_{ij})_{n\times n}=(Z+sI_n)^{-1}B.$ Let $C_1=(|c_{ij}|)_{n\times n}$. Then $$|c_{ij}|=\frac{b_{ij}}{|z_i+s|}\leq \frac{b_{ij}}{s}.$$ By \cite[Theorem~8.1.18]{HornJohnson},  $$\rho(C)\leq\rho(C_1)\leq \rho(B)/s<1.$$ Hence $I-C$ and $Z+A$ are both invertible. 
\end{proof}

\begin{proof}
[Proof of Theorem \ref{thm:nonnegmatrix}] By Lemma \ref{lem:M-matrixCM},
$\forall\x\in\C_{p_{N}},\forall\y\in\mathbb{R}^{n}$, $p_{N}(\x+\sqrt{-1}\y)\not=0.$
As a consequence, $\phi(\mathbf{z})=\frac{1}{p_{N}(\z)}$
is a holomorphic function defined on $\C_{p_{N}}+\sqrt{-1}\mathbb{R}^{n}$.
Then, we apply Bochner's tube theorem to conclude $\text{conv}(\C_{p_{N}})=\C_{p_{N}}$.
By induction on derivatives, we see that $p_{N}$ is indeed an ideal
\gar{} polynomial.
\end{proof}
Assuming Theorem~\ref{thm:mainstructure} for the moment, we obtain the following result.
\begin{cor}
\label{cor:InversionofBH}Let $A$ be an $M$-matrix. Let
\[
q_{A}(x_0,\cdots,x_n):=\det(\mathrm{diag}(x_1,\cdots,x_n)+x_{0}A).
\]
Then $q_{A}$ is a Lorentzian polynomial. 
\end{cor}
\begin{rem}
Brändén--Huh \cite{BrandenHuh20} showed that the inversion of
$q_{A}$ is Lorentzian. This implies that $q_A$ is a dually Lorentzian polynomial introduced by Ross--S\"uss--Wannerer \cite{RossSuessWannerer2025}. Therefore, $q_{A}(\x)$ is Lorentzian and dually Lorentzian.
\end{rem}

These examples show that the class of ideal G{\aa}rding polynomials is
simultaneously broad enough to include fundamental stable and matrix
examples, and rigid enough to encode convexity of the distinguished
positivity component.
\section{Univariate Polynomials and Pitman-Stanley Polytopes}
We analyze univariate \(\gar{}\) polynomials through their
monotone root sequences and relate them to Pitman--Stanley polytopes.
This provides a geometric model for the universal univariate
G{\aa}rding polynomial and yields the first concavity statement needed later.
 
\subsection{The universal MRS component}
First, we recall the definition.
For a nonzero univariate polynomial \(g(x)\) with at least one real root, we
write \(r(g)\) for its largest real root.

\begin{defn}\label{def:MRS}
Let \(f(x)\) be a univariate polynomial of degree \(d\). If each derivative
\(f^{(i)}\), \(0\le i\le d-1\), has a real root, we define the
\emph{root sequence} of \(f\) by
\[
R(f):=\bigl(r(f),\,r(f'),\,\ldots,\,r(f^{(d-1)})\bigr)\in \R^d.
\]
We say that \(f\) has a \emph{monotone root sequence} \emph{(MRS)} if
\[
r(f)\ge r(f')\ge \cdots \ge r(f^{(d-1)}).
\]
\end{defn}
In \cite{Lin23JFA}, Lin called a univariate polynomial \(f(x)\)
\emph{right noetherian} if it satisfies MRS. Clearly, by Definition~\ref{def:garding}, a univariate polynomial $f(x)$ is \gar{}  if and only if $f$ admits a monotone root sequence. See \cite[Section~9.2]{Fang-MaGard}.

We recall a universal construction of univariate \gar{} polynomials from \cite{Fang-MaGard}: Given a decreasing sequence of real numbers
\[\mathbf{r}:=(r_0,r_1,\cdots,r_{d-1}),\;\  r_0\geq r_1 \geq r_2 \geq \cdots \geq r_{d-1}.\]
Let
\(f_{d}(x,\mathbf{r}):=1, \) and for $k=d-1,d-2,\cdots,0$, \[f_k(x,\mathbf{r}):=\int_{r_k}^xf_{k+1}(t,\mathbf{r})dt.\] Then, \begin{equation}
\label{eq:f(x,r)}f_{\mathbf r}(x):=f_0(x,\mathbf{r})=\int_{r_0}^x\int_{r_1}^{t_1}\cdots\int_{r_{d-1}}^{t_{d-1}}dt_d\cdots dt_{1}\end{equation}
is a \gar{} polynomial in $x$ with $R(f)=\mathbf{r}.$  Furthermore, up to a positive constant multiple, a univariate \gar{} polynomial with the root sequence $\mathbf r$ is uniquely determined by \eqref{eq:f(x,r)}.

We introduce some universal objects associated to univariate
\gar{} polynomials.  

For \(d\geq1\), define the moduli chamber of monic degree \(d\)
univariate \gar{} polynomials with MRS as
\begin{equation}\label{MSD_d}
\MRS_d
:=
\left\{
\rr=(r_0,\ldots,r_{d-1})\in\R^d:
r_0\geq r_1\geq\cdots\geq r_{d-1}
\right\}.
\end{equation}

Let $\rhat=(r_{-1},r_0,\ldots,r_{d-1})\in\R^{d+1}$. Inspired by \eqref{eq:f(x,r)}, we define the universal degree \(d\) univariate \gar{} polynomial by $\upoly_0:=1$ and
\begin{equation}\label{upoly-defn}
\upoly_d(\rhat)
:=
\int_{r_0}^{\rminus}
\int_{r_1}^{t_1}
\cdots
\int_{r_{d-1}}^{t_{d-1}}
dt_d\cdots dt_1.
\end{equation}
 
For fixed \(\rr\in\MRS_d\), the restriction $\upoly_d(x,\rr)=f_\rr(x)$ 
is the normalized univariate \gar{} polynomial with root sequence $\rr$ and leading coefficient \(1/d!\). 
See \eqref{eq:f(x,r)}.

Define the universal G{\aa}rding component
\[
\U_d
:=
\left\{
\rhat=(\rminus,r_0,\ldots,r_{d-1})\in\R^{d+1}:
\rminus>r_0\geq r_1\geq\cdots\geq r_{d-1}
\right\}.
\]
We define the  projection
\[
\pi_d:\U_d\longrightarrow\MRS_d,
\qquad
\pi_d(\rminus,r_0,\ldots,r_{d-1})
=
(r_0,\ldots,r_{d-1})
\]
that forgets the fibre coordinate \(\rminus\).  The fibre over
\(\rr=(r_0,\ldots,r_{d-1})\) is the positive ray
\(
\pi_d^{-1}(\rr)=(r_0,\infty),
\)
which is the G{\aa}rding component of the corresponding univariate
\gar{} polynomial $f_{\mathbf r}(x)$ in \eqref{eq:f(x,r)}.

\subsection{Pitman--Stanley polytopes}
Recall the Pitman--Stanley polytope. For
\(\a=(a_1,\ldots,a_d)\in\overline{\Gamma_d^+}\), define
\begin{equation}
\label{new104}
\PS_d(\a)
:=
\left\{
\y\in\overline{\Gamma_d^+}
:\ 
\sum_{i=1}^{k} y_i
\le
\sum_{i=1}^{k} a_i,
\quad
k=1,\ldots,d
\right\}.
\end{equation}
When \(\a\in\Gamma_d^+\), the set \(\PS_d(\a)\) is a \(d\)-dimensional convex
polytope introduced in~\cite{PS}. Equivalently, it admits the Minkowski-sum
representation
\begin{equation}
\label{minkovskisum}
\PS_d(\a)=\sum_{i=1}^{d}a_i\Delta_i,
\qquad
\Delta_i=\mathrm{conv}\bigl(\mathbf{0},\mathbf{e}_i,\mathbf{e}_{i+1},\ldots,\mathbf{e}_d\bigr).   
\end{equation}
Denote by \(\mathrm{Vol}_d\) the \(d\)-dimensional volume function.

We also define polynomials \(P_k\) recursively by
\[
P_0(a_0)=1,
\]
and, for \(k\geq1\),
\begin{equation}
    \label{new103}
P_k(a_1,\ldots,a_k)
=
\int_0^{a_1}
P_{k-1}(a_2+t,a_3,\ldots,a_k)\,dt.
\end{equation}

Thus \(P_k\) is defined on $\overline{\Gamma^+_k}$ and is smooth on
\(\Gamma_k^+\).
\subsection{From the MRS Model to the Pitman--Stanley Model}
We   relate the universal objects of univariate G\aa rding polynomials to the Pitman--Stanley polytopes and their volume functions.
Define the gap map
\begin{align}
\gapmap_d:\U_d &\to \R_{>0}\times\overline{\Gamma^+_{d-1}},\\
\rhat&\mapsto(\rminus-r_0,\ r_0-r_1,\ \ldots,\ r_{d-2}-r_{d-1}). \end{align}
Denote $\a:=A_d(\rhat)$. In
\(
\U_d^\circ
:=
\{\rminus>r_0>r_1>\cdots>r_{d-1}\},
\)
we have \(\gapmap_d(\U_d^\circ)=\Gamma_d^+\).  Together with the
translation coordinate \(s=r_{d-1}\), the map
\(
\rhat\longmapsto(\a,s)
\)
gives an affine trivialization
\[
\U_d^\circ\simeq\Gamma_d^+\times\R_s.
\]

\begin{thm}
\label{thm:PS-realization}We use notation as above.
For every \(\rhat\in\U_d\), 

\[
\upoly_d(\rhat)
=
P_d(\gapmap_d(\rhat))
=
\mathrm{Vol}_d\bigl(\PS_d(\gapmap_d(\rhat))\bigr).
\]
\end{thm}

\begin{proof} 
We first prove  
$\upoly_d(\rhat)=P_d(\a)$ for $\a=\Ad(\rhat)$.

For \(d=1\), this is immediate by definitions:
\[
\upoly_1(r_-,r_0)
=
\int_{r_0}^{r_-}dt
=
r_- - r_0
=
a_1
=
P_1(a_1).
\]
Assume the identity holds in degree \(d-1\).  By \eqref{upoly-defn},
\[
\upoly_d(\rhat)
=
\int_{r_0}^{\rminus}
\upoly_{d-1}(t,r_1,\ldots,r_{d-1})\,dt.
\]
Let
\(
s=t-r_0\). 
Then \(s\in[0,a_1]\), and 
\[
A_{d-1}(t,r_1,\ldots,r_{d-1})=(a_2+s,a_3,\ldots,a_d).
\]
By the induction hypothesis,
\[
\upoly_{d-1}(t,r_1,\ldots,r_{d-1})
=
P_{d-1}(a_2+s,a_3,\ldots,a_d).
\]
Therefore, by \eqref{new103},
\[
\upoly_d(\rhat)
=
\int_0^{a_1}
P_{d-1}(a_2+s,a_3,\ldots,a_d)\,ds
=
P_d(a_1,\ldots,a_d).
\]

It remains to show
\[
P_d(\a)=\mathrm{Vol}_d(\PS_d(\a)).
\]
We argue by induction on \(d\).  For \(d=1\),
\(
P_1(a_1)=a_1=\mathrm{Vol}_1([0,a_1]).
\)

Assume the result holds in dimension \(d-1\).  Let
\(\a=(a_1,\ldots,a_d)\in\Gamma_d^+\).  Consider the projection
\[
\pi_1:\PS_d(\a)\longrightarrow[0,a_1],
\qquad
\y\longmapsto a_1-y_1.
\]
Let \(t=a_1-y_1\).  Then \(y_1=a_1-t\).  For
\((y_1,\ldots,y_d)\in\PS_d(\a)\), the inequalities in \eqref{new104}
become, for \(2\leq k\leq d\),
\[
\sum_{i=2}^{k}y_i
\leq
t+\sum_{i=2}^{k}a_i.
\]
Hence the fibre over \(t\) is
\[
\pi_1^{-1}(t)
=
\{a_1-t\}\times
\PS_{d-1}(a_2+t,a_3,\ldots,a_d).
\]
By the induction hypothesis,
\[
\mathrm{Vol}_{d-1}
\bigl(
\PS_{d-1}(a_2+t,a_3,\ldots,a_d)
\bigr)
=
P_{d-1}(a_2+t,a_3,\ldots,a_d).
\]
Therefore, by Fubini's theorem,
\[
\mathrm{Vol}_d(\PS_d(\a))
=
\int_0^{a_1}
P_{d-1}(a_2+t,a_3,\ldots,a_d)\,dt
=
P_d(\a).
\]
This proves the theorem on \(\Gamma_d^+\), and the identity extends to \(\overline{\Gamma_d^+}\),
 by continuity.
\end{proof}

\begin{cor}
\label{cor:universal-root-concavity}
Let \(d\geq1\). Then the function
\(
        \upoly_d(\rhat)^{\frac{1}{d}}
\)
is concave on the universal G{\aa}rding component \(\U_d\).
\end{cor}

\begin{proof}

First, we prove that \(P_d^{1/d}\) is concave in
the Pitman--Stanley coordinates.

By  \eqref{minkovskisum},
 $\PS_d(\a)$ depends linearly on $\a$ under Minkowski addition, and we apply Brunn--Minkowski theorem to get
\(
        \Vol_d(\PS_d(\a))^{\frac{1}{d}}
        =
        P_d(\a)^{\frac{1}{d}}
\)
is concave with respect to $\a$ on \(\Gamma_d^+\). This concavity extends to
\(\R_{>0}\times\R_{\geq0}^{d-1}\) by continuity.

Next, note that the gap map
\[
        \gapmap_d:\U_d\to \R_{>0}\times\R_{\geq0}^{d-1}
\]
is affine and affine pullbacks preserve concavity, the function
$
        P_d(\gapmap_d(\rhat))^{\frac{1}{d}}
        =
        \upoly_d(\rhat)^{\frac{1}{d}}$
is concave on \(\U_d\).
\end{proof}
 The Pitman--Stanley realization provides the first concavity statement for the
universal model. In the next section, we expand this to quotient
concavity.

\section{Universal Quotient}
We study the universal quotient associated with the
universal univariate \(\gar{}\) polynomial. This quotient will serve as
the model for the fibre quotient \(f/\partial_y f\) that appears later in the
proof of the polarization theorem.

This section forms the technical core of the paper and develops new geometric insight; the level of detail is essential to establish the concavity and monotonicity properties used throughout the sequel.

For $d\geq1$, denote $\rhat=(r_{-1},r_0,\cdots,r_{d-1})$. Let $\upoly_d(\rhat)$ be the universal univariate \gar{} polynomial of degree $d$. 

Denote $\pdvminus:=\frac{\partial}{\partial\rminus}.$ Define
\begin{equation}\label{uquot}
\uquot_d(\rhat)
:=
\frac{\upoly_d(\rhat)}{\pdvminus\upoly_d(\rhat)}.
\end{equation}

For \(d=1\), set
\(
\U_1^{(1)}:=\R^2.
\)
For \(d\ge2\), set
\[
\U_d^{(1)}
:=
\left\{
\rhat=(\rminus,r_0,\ldots,r_{d-1})
:
\rminus>r_1,\;
r_0\ge r_1\ge\cdots\ge r_{d-1}
\right\}.
\]
Then $\uquot_d(\rhat)$ is a rational function on  
$\U_d^{(1)}$. We call $\uquot_d(\rhat)$ the \emph{universal quotient} or \emph {universal quotient function} in degree $d$.

On the region where the denominator is nonzero, define
\[
\Psi_d(\a)
:=
\frac{P_d(\a)}{\partial_1P_d(\a)}
=
\frac{P_d(a_1,\ldots,a_d)}
{P_{d-1}(a_1+a_2,a_3,\ldots,a_d)}.
\]
On \(\Gamma_d^+\), this is the Pitman--Stanley quotient.
Moreover, for \(\rhat\in\U_d\), Theorem~\ref{thm:PS-realization} and the chain rule give
\[
\uquot_d(\rhat)=\Psi_d(\gapmap_d(\rhat)).
\]

The function $\Psi_d$ is homogeneous  degree $1$. Geometrically, it is the ratio of the volume of the Pitman--Stanley polytope to its first variation in the direction of the simplex $\Delta_1$:
\[
\Psi_d(\a)
=
\frac{\Vol_d(\PS_d(\a))}{\partial_1 \Vol_d(\PS_d(\a))}.
\]
Here, $\partial_1 \Vol_d(\PS_d(\a))$ is the mixed volume obtained by replacing one copy of $\PS_d(\a)$ with $\Delta_1$ in the $d$-fold mixed volume expansion of $\Vol_d(\PS_d(\a))$.

We will establish the concavity of $\Psi_d$ in $\a$ and of $\uquot_d(\rhat)$ in $\rhat$ on suitable domains.

\subsection{Preparation lemmas}
We prepare our study with a few technical lemmas.

\begin{lem}
\label{lem:segmentslicing}
Let \(u,v,h\) be continuous positive functions defined on the intervals
$[\alpha,a]$, $[\beta,b]$, $[\alpha+\beta,a+b]$,
respectively. Assume that
\[
h(s+t)\geq u(s)+v(t)
\]
for every \(s\in[\alpha,a]\) and \(t\in[\beta,b]\). Then
\[
\int_{\alpha+\beta}^{a+b}\frac{dz}{h(z)}
\leq
\max\left\{
\int_{\alpha}^{a}\frac{ds}{u(s)},
\int_{\beta}^{b}\frac{dt}{v(t)}
\right\}.
\]
\end{lem}

\begin{proof}
By translating the variables, it suffices to prove the case
\(\alpha=\beta=0\). Set
\[
U:=\int_0^a\frac{ds}{u(s)},
\qquad
V:=\int_0^b\frac{dt}{v(t)}.
\]
Without loss of generality, assume \(V\geq U\).

Let \(\gamma_u:[0,U]\to[0,a]\) solve
\[
\gamma_u'(\tau)=u(\gamma_u(\tau)),
\qquad
\gamma_u(0)=0,
\]
so that \(\gamma_u(U)=a\). Extend \(\gamma_u\) to \([0,V]\) by setting
\(
\gamma_u(\tau)=a\) for \( U\leq\tau\leq V.
\)
Similarly, let \(\gamma_v:[0,V]\to[0,b]\) solve
\[
\gamma_v'(\tau)=v(\gamma_v(\tau)),
\qquad
\gamma_v(0)=0,
\]
so that \(\gamma_v(V)=b\).

Define
\[
\gamma(\tau):=\gamma_u(\tau)+\gamma_v(\tau).
\]
Then \(\gamma\) is increasing and runs from \(0\) to \(a+b\) as
\(\tau\) runs from \(0\) to \(V\). Moreover,
\[
\gamma'(\tau)
\leq
u(\gamma_u(\tau))+v(\gamma_v(\tau))
\leq
h(\gamma(\tau)),
\]
for $\tau\in[0,U)\cup(U,V]$.
Therefore
\[
\int_0^{a+b}\frac{dz}{h(z)}
=
\int_0^V
\frac{\gamma'(\tau)}{h(\gamma(\tau))}\,d\tau
\leq
\int_0^V d\tau
=
V.
\]
Since \(V=\max\{U,V\}\), this proves the claim.
\end{proof}

We also use the following one-dimensional Brascamp--Lieb type inequality;
cf.~\cite[Theorem 3.1]{BRASCAMP1976366}.
\begin{lem}
\label{lem:BLinequality}
Let \(u,v,h\) be non-negative increasing functions defined on
\([0,a]\), \([0,b]\), and \([0,a+b]\), respectively. Assume that $u(a)=v(b)$ and 
\begin{equation}\label{new100}
     h(s+t)\geq \min\{u(s),v(t)\}
\end{equation}
for every \(s\in[0,a]\) and \(t\in[0,b]\). Then
\begin{equation}
        \int_0^{a+b}h(\tau)\,d\tau
        \geq
        \int_0^a u(s)\,ds
        +
        \int_0^b v(t)\,dt .
\label{eq:BLineqresembles}
\end{equation}
\end{lem}

\begin{proof}
Let $T=u(a)=v(b)$. For \(\lambda\geq 0\), define the superlevel sets
\[
        H(\lambda):=\{\tau\in[0,a+b]:h(\tau)>\lambda\},
\]
\[
        U(\lambda):=\{s\in[0,a]:u(s)>\lambda\},
        \qquad
        V(\lambda):=\{t\in[0,b]:v(t)>\lambda\}.
\] 
Since \(u,v,h\) are increasing, the sets
\(U(\lambda)\), \(V(\lambda)\), and \(H(\lambda)\) are intervals.  All 3 sets are not empty if $\lambda< T$ and only $H(\lambda)$ can be non-empty for $\lambda\geq T$.   \eqref{new100} implies that for \(s\in U(\lambda)\) and \(t\in V(\lambda)\), 
\[
        h(s+t)\geq\min\{u(s),v(t)\}>\lambda.
\]
Therefore, \( U(\lambda)+V(\lambda)\subseteq H(\lambda),\) which implies
\begin{equation}\label{new102}
    |H(\lambda)|
        \geq
        |U(\lambda)|+|V(\lambda)|.
\end{equation}

Integrating \eqref{new102} over \(\lambda\in[0,\infty)\), we obtain
\[
        \int_0^\infty |H(\lambda)|\,d\lambda
        \geq
        \int_0^\infty |U(\lambda)|\,d\lambda
        +
        \int_0^\infty |V(\lambda)|\,d\lambda .
\]By the layer-cake representation, this is exactly \eqref{eq:BLineqresembles}.
\end{proof}
\subsection{Concavity of \(\Psi_d\)}
\begin{thm}
\label{thm:superadditive}
The Pitman-Stanley quotient \(\Psi_d\) is concave on \(\Gamma_d^+\). Since \(\Psi_d\) is
homogeneous of degree \(1\), the concavity is equivalent to the superadditivity
\[
\Psi_d(\a+\mathbf b)
\geq
\Psi_d(\a)+\Psi_d(\mathbf b),
\qquad
\forall\,\a,\mathbf b\in\Gamma_d^+.
\label{eq:superadditive}
\]
\end{thm}

\begin{proof} We argue by induction in $d$. 
The base case \(d=1\) is immediate, since
\(P_1(a_1)=\Psi_1(a_1)=a_1.
\)

Assume for any
\(\a,\b\in\Gamma_{d-1}^+\), we have
\[
        \Psi_{d-1}(\a+\b)
        \geq
        \Psi_{d-1}(\a)+\Psi_{d-1}(\b).
\]

Now in degree $d$, let
$\a=(a_1,a_2,\ldots,a_d)\in\Gamma_d^+.$
For \(t\in[0,a_1]\), set
\[
        \a'_t:=(a_2+t,a_3,\ldots,a_d)\in\Gamma_{d-1}^+.
\]
Define
\[
        \psi(t,\a)
        :=
        \frac{P_{d-1}(\a'_t)}
        {P_{d-1}(\a'_{a_1})}.
\]
Then, by the recursive definition of \(P_d\),
\begin{equation}
\label{eq:psiandPsi}
        \Psi_d(\a)
        =
        \int_0^{a_1}\psi(t,\a)\,dt.
\end{equation}
Moreover,
\[
\psi(t,\a)
=
\exp\left(
-\int_t^{a_1}
\frac{\partial_1P_{d-1}(\a'_\tau)}
{P_{d-1}(\a'_\tau)}
\,d\tau
\right),
\] which can be rewritten as
\begin{equation}
\label{eq:psi-exponential}
        \psi(t,\a)
        =
        \exp\left(
        -\int_t^{a_1}
        \frac{d\tau}{\Psi_{d-1}(\a'_\tau)}
        \right).
\end{equation}

Let 
        $\b=(b_1,b_2,\ldots,b_d)\in\Gamma_d^+,$ $
        \c=\a+\b.$
For \(s\in[0,b_1]\), and \(\tau\in[0,a_1+b_1]\) define
\[
        \b'_s:=(b_2+s,b_3,\ldots,b_d),\quad  \c'_\tau:=(c_2+\tau,c_3,\ldots,c_d).
\]
By the induction hypothesis,
\[
        \Psi_{d-1}(\c'_{t+s})
        =
        \Psi_{d-1}(\a'_t+\b'_s)
        \geq
        \Psi_{d-1}(\a'_t)+\Psi_{d-1}(\b'_s).
\]
We   apply Lemma \ref{lem:segmentslicing} with
\[
        h(\tau):=\Psi_{d-1}(\c'_\tau),\quad u(\tau):=\Psi_{d-1}(\a'_\tau),
        \quad
        v(\tau):=\Psi_{d-1}(\b'_\tau),
\]
on the intervals $[t+s,a_1+b_1]$, $[t,a_1]$, and $[s,b_1],$ respectively.
Then Lemma
\ref{lem:segmentslicing} gives
\[
\int_{t+s}^{a_1+b_1}
\frac{d\tau}{\Psi_{d-1}(\c'_\tau)}
\leq
\max\left\{
\int_t^{a_1}\frac{d\tau}{\Psi_{d-1}(\a'_\tau)},
\int_s^{b_1}\frac{d\tau}{\Psi_{d-1}(\b'_\tau)}
\right\}.
\]
Note that by \eqref{eq:psi-exponential}, $\psi$ is  increasing in $t$. Hence,
\begin{align}\label{eq:new101}
    \psi(t+s,\a+\b)
        \geq
        \min\{\psi(t,\a),\psi(s,\b)\}.
\end{align}
Notice that $\psi(a_1,\a)=\psi(b_1,\b)=1$. Apply Lemma \ref{lem:BLinequality} to
\[
        u(t)=\psi(t,\a),
        \qquad
        v(s)=\psi(s,\b),
        \qquad
        h(\tau)=\psi(\tau,\a+\b).
\] Then, we have
\[
\int_0^{a_1+b_1}\psi(\tau,\a+\b)\,d\tau
\geq
\int_0^{a_1}\psi(t,\a)\,dt
+
\int_0^{b_1}\psi(s,\b)\,ds.
\]
By \eqref{eq:psiandPsi}, this is exactly
\[
        \Psi_d(\a+\b)
        \geq
        \Psi_d(\a)+\Psi_d(\b).
\]
Thus \(\Psi_d\) is superadditive on \(\Gamma_d^+\). Since \(\Psi_d\) is
homogeneous of degree \(1\), superadditivity is equivalent to concavity.
This completes the induction.
\end{proof}

\begin{rem}
    For convex bodies $K,L$ in $\R^d$, the mixed volumes $V_k(K,L)$ are defined via the expansion of the volume function: \[\Vol_d(K+tL)=\sum_{k=0}^dV_k(K,L)t^{d-k}.\]Then, we have \[\Psi_d(\a)=\frac{V_d(\PS_d(\a))}{V_{d-1}(\PS_d(\a),\Delta_1)}.\] When $L$ and $B$ are balls in $\R^d$, Giannopoulos-Hartzoulakic-Paouris \cite{GiannopoulosHartzoulakicPaouris} showed \begin{equation}\label{eq:convexsubtle}
        \frac{V_d(K+L)}{V_{d-1}(K+L,B)}\geq\frac{V_d(K)}{V_{d-1}(K,B)}+\frac{V_d(L)}{V_{d-1}(L,B)}.
    \end{equation} This inequality is very similar to Theorem \ref{thm:superadditive}. However, inequality \eqref{eq:convexsubtle} does not hold for general convex bodies. In fact, Fradelizi-Meyer-Giannopoulos \cite{FradeliziMeyerGiannopoulos} argued that  \eqref{eq:convexsubtle}  holds for any triple of convex bodies $(K,L,B)$ only when $d=1,2$. Our proof of Theorem \ref{thm:superadditive} exploits the special structure of Pitman-Stanley polytope which is not available for general convex bodies. 
\end{rem} 

\begin{cor}
\label{cor:convexityconcavity}
Fix \(1\leq j\leq d\).
\begin{enumerate}[label=(\arabic*)]
\item \label{enu:cor1}
For \(\a\in\Gamma_d^+\), define
\[
g(\a):=\frac{P_d(\a)}{\partial_1^jP_d(\a)}.
\]
Then \(g(\a)^{1/j}\) is concave on \(\Gamma_d^+\).

\item \label{enu:cor2}
For \((\a,u)\in\Gamma_{d+1}^+\), define
\[
h(\a,u)
:=
\frac{\partial_1^jP_d(\a)\,u^{j+1}}
     {P_d(\a)}
=
\frac{u^{j+1}}{g(\a)}.
\]
Then \(h\) is convex on $\Gamma_{d+1}^+$.
\end{enumerate}
\end{cor}

\begin{proof}
We first prove \ref{enu:cor1}. For \(0\le l\le j-1\),
\[
\partial_1^lP_d(a_1,\ldots,a_d)
=
P_{d-l}(a_1+\cdots+a_{l+1},a_{l+2},\ldots,a_d).
\]
Hence
\(
\frac{\partial_1^lP_d}{\partial_1^{l+1}P_d}
\)
is the pullback of \(\Psi_{d-l}\) by the linear map
\[
(a_1,\ldots,a_d)
\longmapsto
(a_1+\cdots+a_{l+1},a_{l+2},\ldots,a_d).
\]
By Theorem~\ref{thm:superadditive}, \(\Psi_{d-l}\) is concave on
\(\Gamma_{d-l}^+\). Therefore each quotient
\(
\frac{\partial_1^lP_d(\a)}
     {\partial_1^{l+1}P_d(\a)}
\)
is concave on \(\Gamma_d^+\).  Note
\[
g(\a)^{1/j}
=
\left(
\prod_{l=0}^{j-1}
\frac{\partial_1^lP_d(\a)}
     {\partial_1^{l+1}P_d(\a)}
\right)^{1/j}.
\]
Repeated application of
Lemma~\ref{lem:bootstrapconcavity}
shows that \(g^{1/j}\) is concave on $\Gamma^+_d$ . This proves
\ref{enu:cor1}.

We   prove \ref{enu:cor2}. Set
\(
p(\a):=
\frac{\partial_1^jP_d(\a)}
     {P_d(\a)}.
\)
Since \(p^{-1/j}=g^{1/j}\) is concave on $\Gamma^+_d$, for every
\(\v'\in\R^d\),
\begin{equation}
p\,\mathrm{Hess}\,p(\v',\v')
-
\frac{j+1}{j}
\langle\nabla p,\v'\rangle^2
\ge0.
\label{eq:convexity}
\end{equation}

Since
\(
h(\a,u)=p(\a)u^{j+1},
\)
for \(\v=(\v',v)\in\R^{d+1}\) we compute
\[
\mathrm{Hess}\,h(\v,\v)
=
u^{j-1}
\Bigl(
u^2\mathrm{Hess}\,p(\v',\v')
+
2(j+1)uv\langle\nabla p,\v'\rangle
+
j(j+1)p\,v^2
\Bigr).
\]
For fixed \(\a,u,\v'\), this is a quadratic polynomial in \(v\).
Its discriminant equals
\[
\Delta
=
4j(j+1)u^{2j}
\left(
\frac{j+1}{j}\langle\nabla p,\v'\rangle^2
-
p\,\mathrm{Hess}\,p(\v',\v')
\right).
\]
By \eqref{eq:convexity}, we have \(\Delta\le0\). Since the
coefficient of \(v^2\) is positive,
\[
\mathrm{Hess}\,h(\v,\v)\ge0.
\]
Hence \(h\) is convex on $\Gamma_{d+1}^+$.
\end{proof}

\subsection{Concavity of $\uquot_d$}
Our next result extends the concavity of $\uquot_d$ from $U_d$ to a bigger set. 
\begin{thm}
\label{thm:universalquotientconcavity}
Let \(d\ge1\).  \(\uquot_d(\rhat)\) is concave on $\U_d^{(1)}$, where \[
\U_d^{(1)}
=
\left\{
\rhat=(\rminus,r_0,\ldots,r_{d-1})
:
\rminus>r_1,\;
r_0\ge r_1\ge\cdots\ge r_{d-1}
\right\}.\]
\end{thm}

\begin{proof}
The case \(d=1\) is trivial, since
\(
\uquot_1(\rminus,r_0)=\rminus - r_0.
\)

Assume \(d\ge2\). Let
\(
\rhat=(\rminus,r_0,\ldots,r_{d-1})
\in
\U_d^{(1)}.
\)

We split into two cases: $\rminus>r_0$ and $r_1<\rminus\leq r_0$.

First suppose \(\rminus>r_0\). Then
\[
\a
=
A_d(\rhat)
=
(\rminus-r_0,r_0-r_1,\ldots,r_{d-2}-r_{d-1})
\in
\R_{>0}\times\R_{\ge0}^{d-1}.
\]
On the strict chamber \(\a\in\Gamma_d^+\), 
by Theorem~\ref{thm:superadditive}, \(\Psi_d\) is concave on
\(\Gamma_d^+\).  
Consequently, by affine pullback and continuity up to the relevant boundary,
the universal quotient
\(
\uquot_d=\gapmap_d^*\Psi_d
\)
is concave on \(\U_d\). 

Now suppose $r_1<\rminus\le r_0.$
Set
\[
u:=r_0-\rminus\ge0,
\qquad
b_1:=\rminus-r_1>0,
\]
and
\[
\b
=
(b_1,\ldots,b_{d-1})
=
(\rminus-r_1,r_1-r_2,\ldots,r_{d-2}-r_{d-1})
\in\R_{>0}\times\overline{\Gamma^+_{d-2}}
\]

By the definition of \(\upoly_d\),
\[
\upoly_d(\rhat)
=
-\int_0^u
P_{d-1}(b_1+t,b_2,\ldots,b_{d-1})\,dt,
\]
while
\(
\pdvminus\upoly_d(\rhat)
=
P_{d-1}(\b).
\)
Therefore
\[
\uquot_d(\rhat)
=
-\frac{1}{P_{d-1}(\b)}
\int_0^u
P_{d-1}(b_1+t,b_2,\ldots,b_{d-1})\,dt.
\]

Expanding in the first variable,
\[
P_{d-1}(b_1+t,b_2,\ldots,b_{d-1})
=
\sum_{l=0}^{d-1}
\frac{\partial_1^lP_{d-1}(\b)}{l!}\,t^l.
\]
Hence
\begin{equation}\label{new105}
\uquot_d(\rhat)
=
-\sum_{l=0}^{d-1}
\frac{1}{(l+1)!}
\frac{
u^{l+1}\partial_1^lP_{d-1}(\b)
}{
P_{d-1}(\b)
}.
\end{equation}

For \(l=0\), the summand is simply \(u\), hence linear. For
\(1\le l\le d-1\), Corollary~\ref{cor:convexityconcavity}
implies that the summand in \eqref{new105},
\(
\frac{
u^{l+1}\partial_1^lP_{d-1}(\b)
}{
P_{d-1}(\b)
}
\)
is convex in \((u,\b)\). Therefore, \(\uquot_d\) as in \eqref{new105}, being the negative
of affine linear combination of convex functions, is concave in
the region \(r_1<\rminus\le r_0\).

Since \(\uquot_d\) is smooth on \(\U_d^{(1)}\), the Hessian inequality
extends across the interface \(\rminus=r_0\) by continuity. Hence the
Hessian of \(\uquot_d\) is negative semidefinite throughout the convex
domain \(\U_d^{(1)}\), and therefore \(\uquot_d\) is concave there.
\end{proof}

\subsection{Monotonicity of the universal quotient}

We show that $\uquot_d$ is nonincreasing in the $r_i$ variables. 

\begin{thm}
\label{thm:universalquotientmonotone}
Let \(d\ge 1\). For each \(i=0,\ldots,d-1\), one has
\[
\partial_{r_i}\uquot_d(\rhat)\le 0
\qquad\text{for all }\rhat\in \U_d^{(1)}.
\]
\end{thm}

\begin{proof}
Write
\[
\rhat=(s,\rr)=(s,r_0,\ldots,r_{d-1}).
\]
For \(1\le m\le d\), we use the shorthand
\[
\upoly_m(s,\rr):=\upoly_m(s,r_{d-m},\ldots,r_{d-1}).
\]
Thus
\begin{equation}\label{eq:upoly-partial-minus}
\partial_- \upoly_m(s,\rr)=\upoly_{m-1}(s,\rr),
\qquad 1\le m\le d.
\end{equation}

We first consider differentiation with respect to \(r_0\). Since
\(\upoly_{d-1}(s,r_1,\ldots,r_{d-1})\) is independent of \(r_0\),
\begin{equation}\label{eq:dr0-upoly}
\partial_{r_0}\upoly_d(s,r_0,\ldots,r_{d-1})
=
-\upoly_{d-1}(r_0,r_1,\ldots,r_{d-1}).
\end{equation}
If \(s>r_1\) and \(\rr\in\MRS_d\), then
\begin{equation}\label{eq:upoly-positive}
\upoly_{d-1}(s,r_1,\ldots,r_{d-1})>0.
\end{equation}
Since \(r_0\ge r_1\), \eqref{eq:upoly-positive} also gives
\begin{equation}\label{eq:upoly-nonnegative-at-r0}
\upoly_{d-1}(r_0,r_1,\ldots,r_{d-1})\ge 0.
\end{equation}
Hence, by \eqref{eq:dr0-upoly}--\eqref{eq:upoly-nonnegative-at-r0},
\[
\partial_{r_0}\uquot_d(s,\rr)
=
-\frac{\upoly_{d-1}(r_0,r_1,\ldots,r_{d-1})}
        {\upoly_{d-1}(s,r_1,\ldots,r_{d-1})}
\le 0,
\qquad s>r_1.
\]

Now fix \(i\ge 1\). For \(d-i\le k\le d-1\) and \(s>r_{d-k}\), define
\begin{equation}\label{eq:theta-k}
\theta_k(s)
:=
-\partial_{r_i}\log \upoly_k(s,r_{d-k},\ldots,r_{d-1})
=
-\frac{\partial_{r_i}\upoly_k(s,r_{d-k},\ldots,r_{d-1})}
        {\upoly_k(s,r_{d-k},\ldots,r_{d-1})}.
\end{equation}
We claim that \(\theta_k\) is decreasing on \((r_{d-k},\infty)\).

We prove this by induction on \(k\). First consider \(k=d-i\). Then
\(\upoly_{k-1}\) is independent of \(r_i\), so
\[
\partial_{r_i}\upoly_k(s,r_i,\ldots,r_{d-1})
=
-\upoly_{k-1}(r_i,r_{i+1},\ldots,r_{d-1}),
\]
and therefore
\[
\theta_k(s)
=
\frac{\upoly_{k-1}(r_i,r_{i+1},\ldots,r_{d-1})}
     {\upoly_k(s,r_i,\ldots,r_{d-1})}.
\]
Since \(\upoly_k(s,r_i,\ldots,r_{d-1})\) is increasing in \(s>r_i\),
it follows that \(\theta_k\) is decreasing on \((r_i,\infty)\).

Assume now that \(\theta_{k-1}\) is decreasing on
\((r_{d-k+1},\infty)\). Since \(\rr\in\MRS_d\), we have
\(r_{d-k}\ge r_{d-k+1}\). By \eqref{eq:theta-k} and
\eqref{eq:upoly-partial-minus},
\[
\theta_k(s)
=
\frac{
\displaystyle
\int_{r_{d-k}}^{s}
\theta_{k-1}(t)\,
\upoly_{k-1}(t,r_{d-k+1},\ldots,r_{d-1})\,dt}
{\displaystyle
\int_{r_{d-k}}^{s}
\upoly_{k-1}(t,r_{d-k+1},\ldots,r_{d-1})\,dt}.
\]
Differentiating gives
\[
\theta_k'(s)
=
\frac{
\upoly_{k-1}(s,r_{d-k+1},\ldots,r_{d-1})}
{\upoly_k(s,r_{d-k},\ldots,r_{d-1})^2}
\int_{r_{d-k}}^{s}
\upoly_{k-1}(t,r_{d-k+1},\ldots,r_{d-1})
\bigl(\theta_{k-1}(s)-\theta_{k-1}(t)\bigr)\,dt.
\]
For \(r_{d-k}<t\le s\), the induction hypothesis gives
\(\theta_{k-1}(s)\le \theta_{k-1}(t)\). Hence
\[
\theta_k'(s)\le 0,
\qquad s>r_{d-k}.
\]
This proves the claim.

We   compute \(\partial_{r_i}\uquot_d\) for \(i\ge 1\). Since
\[
\uquot_d(s,\rr)
=
\frac{\upoly_d(s,r_0,\ldots,r_{d-1})}
     {\upoly_{d-1}(s,r_1,\ldots,r_{d-1})},
\]
we obtain
\[
\begin{aligned}
\partial_{r_i}\uquot_d(s,\rr)
&=
\frac{1}{\upoly_{d-1}(s,r_1,\ldots,r_{d-1})}
\int_{r_0}^{s}
\partial_{r_i}\upoly_{d-1}(t,r_1,\ldots,r_{d-1})\,dt \\
&\qquad
-
\frac{\upoly_d(s,r_0,\ldots,r_{d-1})}
     {\upoly_{d-1}(s,r_1,\ldots,r_{d-1})^2}
\partial_{r_i}\upoly_{d-1}(s,r_1,\ldots,r_{d-1}) \\
&=
\frac{1}{\upoly_{d-1}(s,r_1,\ldots,r_{d-1})}
\int_{r_0}^{s}
\upoly_{d-1}(t,r_1,\ldots,r_{d-1})
\bigl(\theta_{d-1}(s)-\theta_{d-1}(t)\bigr)\,dt.
\end{aligned}
\]

If \(s>r_0\), then \(r_0\le t\le s\), and since \(\theta_{d-1}\) is
decreasing, we get
\[
\partial_{r_i}\uquot_d(s,\rr)\le 0.
\]

It remains to consider the case \(r_1<s\le r_0\). Reversing the orientation
of the integral gives
\[
\partial_{r_i}\uquot_d(s,\rr)
=
\frac{1}{\upoly_{d-1}(s,r_1,\ldots,r_{d-1})}
\int_{s}^{r_0}
\upoly_{d-1}(t,r_1,\ldots,r_{d-1})
\bigl(\theta_{d-1}(t)-\theta_{d-1}(s)\bigr)\,dt.
\]
For \(s\le t\le r_0\), the monotonicity of \(\theta_{d-1}\) yields
\(\theta_{d-1}(t)-\theta_{d-1}(s)\le 0\). Hence
\[
\partial_{r_i}\uquot_d(s,\rr)\le 0.
\]

Therefore, for every \(0\le i\le d-1\),
\[
\partial_{r_i}\uquot_d(s,\rr)\le 0
\qquad\text{throughout }\U_d^{(1)}.
\]
This completes the proof.
\end{proof}

We have therefore established both concavity and monotonicity for the
universal quotient on its natural domain. These two properties are the key
inputs for the fibrewise quotient concavity result proved in the next section.

\section{Fibre Quotient Concavity}
We transfer the universal quotient theory to general ideal
G\aa{}rding polynomials by applying it fibrewise in one distinguished
variable. The main result is a concavity theorem for the quotient
\(f/\partial_y f\) on the positivity component of \(\partial_y f\).

We record a technical composition lemma.
\begin{lem}
\label{lem:concave-decreasing-composition}
Let \(Q(s,r_0,\ldots,r_{d-1})\) be concave on a convex set
\(\Omega\subset\R^{d+1}\).
Assume that \(Q\) is nonincreasing in each \(r_i\)-coordinate, i.e.\,
\[
\partial_{r_i}Q\le0,
\qquad
0\le i\le d-1.
\]
Let \(D\subset\R^n\) be convex, and let \(
\rho_i:D\to\R,\) \(
0\le i\le d-1,
\)
be convex functions. Suppose that
\(
(s,\rho_0(\x),\ldots,\rho_{d-1}(\x))\in\Omega
\)
for all \((\x,s)\) in a convex set \(E\subset D\times\R\). Then
\[
(\x,s)\longmapsto
Q(s,\rho_0(\x),\ldots,\rho_{d-1}(\x))
\]
is concave on \(E\).
\end{lem}

\begin{proof}
Take \((\x_0,s_0),(\x_1,s_1)\in E\), and let \(0\le\lambda\le1\). Set
\[
\x_\lambda=(1-\lambda)\x_0+\lambda\x_1,
\qquad
s_\lambda=(1-\lambda)s_0+\lambda s_1.
\]
Since \(Q\) is concave and decreasing in each \(r_i\)-coordinate, and $\rho_i$ are convex,
\[
\begin{aligned}
&Q(s_\lambda,\rho_0(\x_\lambda),\ldots,\rho_{d-1}(\x_\lambda))
\\
&\ge
Q\Bigl(
s_\lambda,
(1-\lambda)\rho_0(\x_0)+\lambda \rho_0(\x_1),
\ldots,
(1-\lambda)\rho_{d-1}(\x_0)+\lambda \rho_{d-1}(\x_1)
\Bigr)\\
&\geq (1-\lambda)
Q(s_0,\rho_0(\x_0),\ldots,\rho_{d-1}(\x_0))+
\lambda
Q(s_1,\rho_0(\x_1),\ldots,\rho_{d-1}(\x_1)).
\end{aligned}
\]
This proves the claim.
\end{proof}

\begin{lem}
\label{lem:fibre-root-functions-convex}
Let \(f(\x,y)\in\I[\x,y]\), let \(d=\deg_y f\ge1\), and set
\[
D_i:=\pi_\x(\C_{\partial_y^i f}),\qquad 0\le i\le d-1.
\]
For \(\x\in D_i\), let \(r_i(\x)\) be the largest root of
\(\partial_y^i f(\x,y)\). Then \(r_i:D_i\to\R\) is convex, and
\[
\C_{\partial_y^if}
=
\{(\x,y):\x\in D_i,\ y>r_i(\x)\}.
\]
\end{lem}
\begin{proof}
By Corollary~\ref{cor:basicrestriction}, for every \(\x\in D_i\), the
univariate polynomial \(y\mapsto \partial_y^i f(\x,y)\) is G{\aa}rding.
Hence it has largest root \(r_i(\x)\), and
\[
\C_{\partial_y^if}
=
\{(\x,y):\x\in D_i,\ y>r_i(\x)\}.
\]
Since \(f\in\I\), the component \(\C_{\partial_y^if}\) is convex. Therefore
this epigraph is convex, and hence \(r_i\) is convex on \(D_i\).
\end{proof}

\begin{thm}
\label{prop:fibre-quotient-concavity}
Assume \(f(\x,y)\in\I[\x,y]\) and \(\deg_yf\ge1\). Then
\[
h(\x,y):=\frac{f(\x,y)}{\partial_yf(\x,y)}
\]
is concave on \(\C_{\partial_yf}\).
\end{thm}
 \begin{proof}
Let
\[
D:=\pi_{\x}(\C_{\partial_yf}),
\qquad
d:=\deg_y f.
\]
For \(\x\in D\), set \(\phi_\x(y):=f(\x,y)\). By
Corollary~\ref{cor:basicrestriction}, \(\phi_\x\in\GS[y]\). 

Moreover, \(\deg\phi_\x=d\) on \(D\). Indeed, write
\(f(\x,y)=a_d(\x)y^d+\cdots+a_0(\x)\). Since
\(\partial_y^d f=d!a_d(\x)\) and
\(\C_{\partial_yf}\subset\C_{\partial_y^d f}\), every
\(\x\in D\) satisfies \(a_d(\x)>0\). Hence \(\deg\phi_\x=d\).

Write the root sequence of \(\phi_\x\) as
\[
R(\phi_\x)=(r_0(\x),\ldots,r_{d-1}(\x)).
\]
Then
\[
\C_{\partial_yf}
=
\{(\x,y):\x\in D,\ y>r_1(\x)\}.
\]
Moreover, by Lemma~\ref{lem:fibre-root-functions-convex}, the functions
\(r_i\) are convex on their natural domains. The recursive inclusions imply
that \(D\subseteq D_i:=\pi_\x(\C_{\partial_y^if})\) for \(i\ge1\). 

Since \(\phi_\x\in\GS[y]\) and \(\deg\phi_\x=d\), the whole root sequence
\(r_0(\x),\ldots,r_{d-1}(\x)\) is defined. In particular, \(r_0(\x)\ge
r_1(\x)\). Choosing \(Y>r_0(\x)\) gives \((\x,Y)\in\C_f\), and hence
\(\x\in D_0\).
Therefore,
each \(r_i\) is convex on \(D\).

For \((\x,y)\in\C_{\partial_yf}\), define
\[
\rhat(\x,y):=(y,r_0(\x),\ldots,r_{d-1}(\x)).
\]
Since \(y>r_1(\x)\), we have \(\rhat(\x,y)\in\U_d^{(1)}\). By the
universal representation of univariate \(\gar{}\) polynomials,
\[
f(\x,y)
=
c(\x)\upoly_d(\rhat(\x,y)),
\qquad
\partial_yf(\x,y)
=
c(\x)\pdvminus\upoly_d(\rhat(\x,y)),
\]
where \(c(\x)>0\) is the leading normalization factor. Hence
\[
h(\x,y)
=
\frac{f(\x,y)}{\partial_yf(\x,y)}
=
\uquot_d(\rhat(\x,y)).
\]

By Theorem~\ref{thm:universalquotientconcavity}, \(\uquot_d\) is concave
on \(\U_d^{(1)}\), and by Theorem~\ref{thm:universalquotientmonotone},
it is nonincreasing in each root coordinate. Since the functions
\(r_i\) are convex on \(D\), Lemma~\ref{lem:concave-decreasing-composition}
applies and gives the concavity of \(h\) on \(\C_{\partial_yf}\).
\end{proof}
Theorem~\ref{prop:fibre-quotient-concavity} is the main analytic input in the
polarization argument. It allows us to pass from the the convexity of fibers to
the convexity of the polarized positivity component.

\section{Polarization of Ideal \gar{} Polynomials}

We prove that ideal G{\aa}rding polynomials are preserved
under \(\kappa\)-polarization by decomposing \(\kappa\)-polarization into
partial polarizations. The main result of this section is the following

\begin{thm}\label{thm:polarizingidealgarding}
Let \(\kappa\in\N^n\), and assume that
\(f(\x)\in\I_\kappa[\x]\). Then its \(\kappa\)-polarization is ideal G\aa{}rding, i.e.\,
\[
\pol_\kappa f
(x_{11},\ldots,x_{1\kappa_1},\ldots,
 x_{n1},\ldots,x_{n\kappa_n})\in \Ia.
\]
\end{thm}
 
A univariate \gar{} polynomial is ideal \gar{}. By Theorem \ref{thm:polarizingidealgarding}, we obtain the following result.

\begin{thm}\label{thm:polarizingMRS}
    Suppose $d\leq n$. Let $\x=(x_1,\cdots,x_n)$. Let \[F(\x):=\sum_{i=0}^d c_i\frac{\sigma_i(\x)}{\tbinom{n}{i}},\qquad f(t):=\sum_{i=0}^dc_it^i.\] Then, the following are equivalent:
    \begin{enumerate}[label=(\arabic*)]
        \item $F\in\I[\x]$; \label{enu:symmetricIa}
        \item $F\in \GS[\x]$;\label{enu:symmetricGa}
        \item $f(t)$ has a MRS.\label{enu:symmetricMRS}
    \end{enumerate}
\end{thm}

\subsection{Partial polarization}

\begin{defn}\label{ppol}
Let \(m\ge1\). The partial polarization operator
\[
\ppol_m^y:\R_m[y]\longrightarrow\R_{(m,1)}[y,z]
\]
is the linear map determined by
\[
y^k
\longmapsto
\frac{k}{m}y^{k-1}z
+
\frac{m-k}{m}y^k,
\qquad
0\le k\le m.
\]
Extending coefficientwise, for every \(\kappa\in\N^n\) we obtain a
linear map
\[
\ppol_m^y:
\R_{\kappa\oplus m}[\x,y]
\longrightarrow
\R_{\kappa\oplus m\oplus1}[\x,y,z].
\]
Equivalently, for every
\(f(\x,y)\in\R_{\kappa\oplus m}[\x,y]\),
\[
\ppol_m^y f(\x,y,z)
=
f(\x,y)
+
\frac{z-y}{m}\,\partial_y f(\x,y).
\]
\end{defn}

We list some basic properties of the partial polarization.

\begin{lem}
\label{lem:partialppolbasics}
Let \(m\geq1\), let \(\kappa\in\N^n\), and let
\(
f(\x,y)\in\R_{\kappa\oplus m}[\x,y].\)
Then the partial polarization operator \(\ppol_m^y\) satisfies the following.

\begin{enumerate}[label=(\arabic*)]

\item \label{enu:partialppol-diagonal}

\(
\ppol_m^y f(\x,y,y)=f(\x,y).
\)

\item \label{enu:partialppol-degree}
If \(\deg_y f<m\), then
\(
\deg_y\ppol_m^y f=\deg_y f.
\)
If \(\deg_y f=m\), then
\(
\deg_y\ppol_m^y f=m-1.
\)
\item \label{enu:partialppolpolarrestrict}
If
\(
f(\x,y)=\sum_{j=0}^{m}g_j(\x)y^j,
\)
then
\[
\ppol_m^y f(\x,y,z)
=
\sum_{j=0}^{m}
\frac{g_j(\x)}{\tbinom{m}{j}}
\sigma_j(\overbrace{y,\ldots,y}^{m-1},z).
\]

\item \label{enu:partialpolar-usual}
Let \(f(\x)=f(x_1,\ldots,x_n)\in\R_\kappa[\x]\). Then the usual
\(\kappa\)-polarization is obtained by iterated partial polarizations:
\begin{equation}
\label{eq:partialpolarizationto_usual}
\pol_\kappa f
=
\prod_{q=1}^{n}
\left(
\ppol^{x_q}_1\circ\cdots\circ\ppol^{x_q}_{\kappa_q}
\right)f.
\end{equation}
\end{enumerate}
\end{lem}
\begin{proof}
First two claims follow directly from Definition~\ref{ppol}.

 \ref{enu:partialppolpolarrestrict} follows from the definition and the following
\[
\sigma_j(\overbrace{y,\ldots,y}^{m-1},z)
=
\tbinom{m-1}{j}y^j
+
\tbinom{m-1}{j-1}y^{j-1}z.
\]

It remains to prove \ref{enu:partialpolar-usual}. Consider the right hand side of
\eqref{eq:partialpolarizationto_usual}. If
\(i\neq j\), then partial polarizations in the \(x_i\)- and \(x_j\)-variables
commute on products:
\[
\ppol^{x_i}_k\circ\ppol^{x_j}_\ell
\bigl(f(x_i)g(x_j)\bigr)
=
(\ppol^{x_i}_k f)(\ppol^{x_j}_\ell g).
\]
Thus it suffices to prove the identity for a univariate monomial.

Fix \(m\in\N\), and let \(0\leq k\leq m\). For \(0\leq l\leq m\), define
\[
f_l(x,z_1,\ldots,z_l)
:=
\ppol^{x}_{m-l+1}\circ\cdots\circ \ppol^{x}_{m} x^k,
\]
where each application of \(\ppol\) introduces a new variable \(z_r\).
We claim that
\begin{equation}
\label{eq:f_l}
f_l(x,z_1,\ldots,z_l)
=\frac1{\tbinom{m}{k}}
\sum_{j=0}^{l}
\tbinom{m-l}{k-j}
x^{k-j}\sigma_j(z_1,\ldots,z_l),
\end{equation}
with the convention that \(\tbinom{a}{b}=0\) if \(b<0\) or \(b>a\).

For \(l=0\), \eqref{eq:f_l} is trivially true.

Assume \eqref{eq:f_l} holds for some \(l<m\). Note
\[
\ppol^{x}_{m-l}(x^{k-j})
=
\frac{k-j}{m-l}x^{k-j-1}z_{l+1}
+
\frac{m-l-k+j}{m-l}x^{k-j}.
\] Denote \(
\sigma_j=\sigma_j(z_1,\ldots,z_l)\). Then
\[
\begin{aligned}
f_{l+1}
&=
\frac1{\tbinom{m}{k}}\sum_{j=0}^{l}
\tbinom{m-l}{k-j}
\left(
\frac{k-j}{m-l}x^{k-j-1}z_{l+1}\sigma_j
+
\frac{m-l-k+j}{m-l}x^{k-j}\sigma_j
\right)
\\
&=
\frac1{\tbinom{m}{k}}\sum_{j=0}^{l+1}
\tbinom{m-l-1}{k-j}
x^{k-j}
\left(
z_{l+1}\sigma_{j-1}
+
\sigma_j
\right).
\end{aligned}
\]
Since $\sigma_j+\sigma_{j-1}z_{l+1}=\sigma_j(z_1,\cdots,z_{l+1})$,
this completes the induction and proves the claim.

Taking \(l=m\), the only nonzero term in \eqref{eq:f_l} occurs when
\(j=k\). Hence
\[
f_m(x,z_1,\ldots,z_m)
=
\frac{1}{\tbinom{m}{k}}\sigma_k(z_1,\ldots,z_m).
\]
This is exactly the usual polarization of \(x^k\). Therefore the iterated
partial polarization agrees with the usual polarization on all monomials,
and hence on all polynomials by linearity.
\end{proof}

\begin{lem}\label{lem:ppolpreservegar}
    If $f(\x,y)\in\GS_{\kappa\oplus m}[\x,y]$, then $\ppol^y_mf(\x,y,z)\in\GS_{\kappa\oplus m\oplus1}[\x,y,z]$.
\end{lem}
\begin{proof}
By Lemma \ref{lem:partialppolbasics} \ref{enu:partialppolpolarrestrict}, the  $\ppol^y_m$ is a composition of a polarization and a pull back of a strictly positive linear map. By \cite[Theorem 8.3]{Fang-MaGard}, both operations preserve \gar{} polynomials. Thus $\ppol^y_mf$ is also \gar{}. 
\end{proof}

\subsection{Polarizing ideal \gar{} polynomials}
We proceed to show that the partial polarization of an ideal \gar{} polynomial remains an ideal \gar{} polynomial.

\begin{prop}
\label{prop:ppol}
Assume \(f(\x,y)\in\I[\x,y]\) and \(0\le \deg_y f\le m\). Then
\[
\ppol_m^y f(\x,y,z)\in\I[\x,y,z].
\]
\end{prop}

\begin{proof}
Set
\[
F:=\ppol_m^y f
=
f+\frac{z-y}{m}\partial_y f.
\]
We argue by induction on the total degree \(k:=\deg f\).

If \(k\le 1\), then \(F\in\GS\) by Lemma~\ref{lem:ppolpreservegar}. Since
\(\I^0=\GS^0\) and \(\I^1=\GS^1\), it follows that \(F\in\I\).

Assume now \(k\ge 2\), and suppose that the statement is known in all smaller
total degrees. By Lemma~\ref{lem:ppolpreservegar}, we have
\(
F\in\GS[\x,y,z].
\)

We first verify that all nonzero first derivatives of \(F\) are ideal
G{\aa}rding. For every \(i\), since \(\deg(\partial_{x_i}f)<k\), by induction,
\[
\partial_{x_i}F=\ppol_m^y(\partial_{x_i}f)\in\I
\]
If \(\deg_y f\ge 1\), then
\(
\partial_zF=\frac1m\,\partial_y f\in\I.
\)
Moreover,
\[
\partial_yF
=
\frac{m-1}{m}\,\partial_y f+\frac{z-y}{m}\,\partial_y^2 f.
\]
If \(\deg_y f=1\), then \(\partial_y^2 f=0\), so
\(
\partial_yF=\frac{m-1}{m}\,\partial_y f\in\I.
\)\\
If \(\deg_y f\ge 2\),  since \(\deg(\partial_y f)<k\), by induction
\(
\partial_yF
=
\frac{m-1}{m}\,\ppol_{m-1}^y(\partial_y f)\in\I.
\)

Thus, to prove \(F\in\I\), it remains to show that \(\C_F\) is convex.

If \(\deg_y f=0\), then \(F=f\), viewed as a polynomial independent of \(y,z\),
so the claim is immediate.

If \(\deg_y f\ge 1\), then
\[
F
=
\frac{\partial_y f}{m}
\left(
z-y+m\frac{f}{\partial_y f}
\right),
\]
and
\[
\C_F=
\left\{
(\x,y,z):(\x,y)\in\C_{\partial_y f},\ 
z>y-m\frac{f(\x,y)}{\partial_y f(\x,y)}
\right\}.
\]
By Theorem~\ref{prop:fibre-quotient-concavity}, the quotient
\(
\frac{f}{\partial_y f}
\)
is concave on \(\C_{\partial_y f}\). Therefore, the function
\[
(\x,y)\longmapsto y-m\frac{f(\x,y)}{\partial_y f(\x,y)}
\]
is convex on \(\C_{\partial_y f}\). Since \(\C_{\partial_y f}\) is convex,
it follows that \(\C_F\) is the epigraph of a convex function over a convex
base and therefore is convex.

We have proved that \(F\in\GS\), that every nonzero first derivative of \(F\)
belongs to \(\I\), and that \(\C_F\) is convex. Therefore \(F\in\I\).
\end{proof}

\begin{proof}[Proof of Theorem \ref{thm:polarizingidealgarding}]
By Proposition \ref{prop:ppol}, the partial polarization preserves ideal \gar{} polynomials. By Lemma \ref{lem:partialppolbasics}, $\kappa$-polarization is generated by the partial polarization.  Therefore, $\kappa$-polarization also preserves ideal \gar{} polynomials.
\end{proof}

\begin{proof}[Proof of Theorem \ref{thm:polarizingMRS}] \ref{enu:symmetricIa} $\Rightarrow$ \ref{enu:symmetricGa} is obvious. Since the diagonal specialization preserves \gar{} polynomial (Theorem \ref{thm:maps-preserve-garding}), we have \ref{enu:symmetricGa}  $\Rightarrow$ \ref{enu:symmetricMRS}.   

Suppose that $f(t)$ has MRS. Then $f$ is univariate \gar{} and  $\C_{f^{(k)}}=(r(f^{(k)}),+\infty)$ is convex. Therefore, Theorem \ref{thm:polarizingidealgarding} implies \[F(\x)=\pol_nf(\x)\in\I[\x].\] Hence,  \ref{enu:symmetricMRS} $\Rightarrow$ \ref{enu:symmetricIa}.
\end{proof}

\section{Binary Relations}
We introduce the binary relations of ideal domination and
ideal position. These relations refine the corresponding notions for
\(\gar{}\) polynomials by incorporating the convexity of quotient functions on
the distinguished positivity component.

\subsection{Definitions and a dichotomy theorem}

We begin by introducing two binary relations that encode the two alternatives
in the restriction dichotomy.
\begin{defn}
\label{def:dominatesIGA}
Let \(g\in\I_n\), and let \(f\in\R[\x]\) be a nontrivial polynomial.
We say that \(g\) \emph{ideally dominates} \(f\), and write
\[
f\bvtl g,
\]
if for every multi-index \(\alpha\) such that
\(\partial^\alpha f\not\equiv0\) and \(\partial^\alpha g\not\equiv0\), one has
\[
\C_{\partial^\alpha g}\subset\{\partial^\alpha f>0\},
\]
and the quotient
\[
\phi_\alpha(\x)
:=
\frac{\partial^\alpha f(\x)}{\partial^\alpha g(\x)}
\]
is convex on \(\C_{\partial^\alpha g}\). By convention, \(0\bvtl g\) for
every \(g\in\I\).
\end{defn}
\begin{rem}
The ideal domination condition refines the G{\aa}rding domination relation $f \vtl g$ defined in~\cite{Fang-MaGard} by incorporating convexity features. In particular, $f$ need not be G{\aa}rding, which is crucial for later applications.
\end{rem}

\begin{example}

\label{exmp:bvtleg}
For $f(\x)\in\R_{+}^{\mathfrak{a}}[\x]$, and $\sigma_{n}(\x)=\prod_{i=1}^{n}x_{i},$
then $f(\x)\bvtl\sigma_{n}(\x)$. This construction has been used essentially
in our previous work \cite{FANG2024109867}. In general, if $f(\x)\in\R_{\kappa,+}[\x]$, then $f(\x)\bvtl \x^\kappa$. To see that, we write \[f(\x)=\sum_{\alpha\leq\kappa}c_\alpha\x^\alpha,\quad c_\alpha\geq0.\] Notice that for each multi-index $\alpha\leq \kappa$, \(\x\mapsto \x^{\alpha-\kappa} \)    is a convex function in $\Gamma_n^+$. This implies that $f(\x)/\x^\kappa$ is convex in $\Gamma_n^+$. Since $\pdv^{\beta}f(\x)\in \R_{\kappa-\beta,+}[\x]$,  $\pdv^\beta f/\pdv^\beta x^{\kappa}$ is also convex. By definition,  $f \bvtl \x^\kappa$.
\end{example}
  
\begin{defn}
\label{def:properposition}For nontrivial $f(\textbf{x}),g(\textbf{x})\in\I[\x]$,
a pair $(f,g)$ is called in \emph{ideal position}, written as $f\ip g$,
if 
\[
h(\mathbf{x},y)=f(\textbf{x})y+g(\textbf{x})\in\I[\x,y].
\]
 We adopt the convention that $0\ip g\ip0$ for any $g\in\I$. 
\end{defn}

Recall, a pair of \gar{} polynomial $(f(\x),g(\x))$ is called in proper
position if $f(\x)y+g(\x)\in\GS[\x,y]$~\cite{Fang-MaGard}.
 
\begin{rem}\label{rem:iptoidealdominate}
For non-trivial $f,g\in\I[\x]$, $f\ip g$ if and only if $f\prec g$
and 
\(
\varphi_{\alpha}(\x):=\frac{\pdv^{\alpha}g(\x)}{\pdv^{\alpha}f(\x)}
\)
is a concave function in $\C_{\pdv^{\alpha}f}$ whenever $\pdv^{\alpha}f\not\equiv0$. Moreover, $\varphi_\alpha$ is positive in $\C_{\pdv^\alpha g}$. Then since $\varphi_\alpha$ is log-concave in $\C_{\pdv^\alpha g}$, $\frac{\pdv^\alpha f}{\pdv^\alpha g}$ is log-convex and \emph{a fortiori} convex in  $\C_{\pdv^\alpha g}$. 
Therefore, $f\ip g$
implies $f\bvtl g$. 
\end{rem}

The proof of the following Theorem is identical to similar results in~\cite{Fang-MaGard}. We
skip the proof.

\begin{thm}
\label{thm:Suppose-that-is}
Let $h(\x,y)=f(\x)y+g(\x)$ with $f,g$ be nontrivial. Then $h(\x,y)\in \I_{n+1}[\x,y]$ if and only if exactly one of the following holds:
\begin{enumerate}
\item \label{enu:-is-Garding} $f\in \I$ and $-g \bvtl f$;
\item \label{enu:or--are} $f,g\in \I$ and $f \ip g$.
\end{enumerate}
\end{thm}

\subsection{Properties of binary relations}
We record basic properties of ideal domination and ideal position. 
These properties will be used in the sequel as structural tools,
especially in the argument for linear preservers.

The proofs of the remaining results are essentially the same as in~\cite{Fang-MaGard} and omitted. See \cite[section 6, section 10.2]{Fang-MaGard} 

\begin{lem}
\label{lem:basic-domination}Let $g\in\I$. 
\begin{enumerate}
\item \label{enu:If--anddominate}If $f\bvtl g$ and $g\bvtl f$, then $g$
and $f$ are proportional.
\item \label{enu:If--anddominate-1}If $f\bvtl g$, then for any multi-index
$\alpha$, $\pdv^{\alpha}f\bvtl\pdv^{\alpha}g$. 
\item \label{enu:If--anddominate-1-1}If $f\bvtl g$, then $\deg f\leq\deg g$.
\item \label{enu:convexcombinationdomiation}Suppose that $f_{1},f_{2}\in\R[\x]$
and $f_{i}\bvtl g$ for each $i$. For $c_{1},c_{2}\geq0$, $\sum_{i=1}^{2}c_{i}f_{i}\bvtl g$. 
\end{enumerate}
\end{lem}

\begin{prop}
\label{thm:dominategardnew} Assume that $f\bvtl g$. There exists
a constant $\rho_{g}(f)>0$ such that $g-cf\in\I$ for every $c\in[0,\rho_{g}(f))$.
The constant $\rho_{g}(f)$ has an explicit formula:
\[
\rho_{g}(f):=\sup\{c>0:\{g-cf>0\}\cap\C_{g}\neq\emptyset\}.
\]
\end{prop}
 
\begin{lem}
\label{lem:proper-position-basic}The following properties of ideal
position relation holds:
\begin{enumerate}
\item \label{enu:111}If $f\ip g$, then $f\bvtl g$.
\item \label{enu:333}If $f\ip g$, then $\partial^{\alpha}f\ip\partial^{\alpha}g$ for every multi-index \(\alpha\).
\item \label{enu:444}If $f\ip g$, then $\deg f\leq\deg g\leq\deg f+1$.
\item \label{enu:555}If $f\ip g$, and $c\geq0,$ then $cf+g\in\I$.
\end{enumerate}
\end{lem}

\begin{prop}
\label{lem:subordinatetoboth} Suppose that $f,g_{1},g_{2},\cdots,g_{k}\in\I\backslash\{0\}$
and $f(\textbf{x})\ip g_{i}(\textbf{x})$ for all $i\in[k]$. Let
$g=\sum_{i=1}^{k}\theta_{i}g_{i}$ for any $\theta_{i}\geq0$. Then
$g\in\I$. Moreover, 
\begin{equation}
f(\textbf{x})\ip g(\mathbf{x}),\label{eq:dominateconvexcombination}
\end{equation}
if $g$ is non-trivial. 
\end{prop}
\section{Multi-affine Linear Preservers}

We prove the multi-affine linear preserver theorem for
ideal G{\aa}rding polynomials, in the spirit of the corresponding results
in~\cite{BB09I,BrandenHuh20,Fang-MaGard}.

We record the following elementary elimination lemma.

\begin{lem}
\label{lem:linear_preserver_lem}
Suppose \(H(u,w,\z)\in \I_+[u,w,\z]\) is multi-affine in \(u,w\). Then
\[
\partial_u H(0,0,\z)+\partial_w H(0,0,\z)\in \I_+[\z].
\]
\end{lem}

\begin{proof}
Write
\[
H(u,w,\z)=h_{11}(\z)uw+h_{10}(\z)u+h_{01}(\z)w+h_{00}(\z).
\]
Since \(H\in\I_+\), its derivatives
\(
\partial_u H=h_{11}w+h_{10}\) and \(
\partial_w H=h_{11}u+h_{01}
\)
belong to \(\I_+\). Hence
\(
h_{11}\ip h_{10}\) and \(
h_{11}\ip h_{01}.
\)
The zero cases are covered by our standing conventions. By
Proposition~\ref{lem:subordinatetoboth},
\[
\partial_u H(0,0,\z)+\partial_w H(0,0,\z)=h_{10}+h_{01}\in \I_+.
\] The proof is finished.
\end{proof}

\begin{thm}
\label{thm:ma-linear-preserver}
Let
\[
S:\R^{\mathfrak a}[\w]\longrightarrow \R[\z]
\]
be a linear operator, where \(\w=(w_1,\ldots,w_n)\). Define its symbol by
\[
\Sym_S(\z,\v):=
S\left(\prod_{i=1}^n (w_i+v_i)\right),
\qquad
\v=(v_1,\ldots,v_n).
\]
If \(\Sym_S(\z,\v)\in \I_+[\z,\v]\), then
\[
S\bigl(\I_+^{\mathfrak a}[\w]\bigr)\subseteq \I_+[\z].
\]
\end{thm}

\begin{proof}We follow the proof of Theorem 7.2 of \cite{Fang-MaGard}.
Let \(f(\mathbf t)\in \I_+^{\mathfrak a}[\mathbf t]\), where
\(\mathbf t=(t_1,\ldots,t_n)\) is a second set of variables. Set
\[
G(\z,\v,\mathbf t):=\Sym_S(\z,\v)\,f(\mathbf t).
\]
Since variables \((\z,\v)\) and \(\mathbf t\) are disjoint,
Lemma~\ref{lem:external-products} gives
\[
G\in \I_+[\z,\v,\mathbf t].
\]
Moreover, \(G\) is multi-affine in each pair \((v_i,t_i)\), since
\(\Sym_S(\z,\v)\) is multi-affine in \(\v\) and \(f\) is multi-affine in
\(\mathbf t\). Apply Lemma~\ref{lem:linear_preserver_lem} successively to the pairs
\((v_i,t_i)\), \(i=1,\ldots,n\). Then, we have
\[\sum_{\alpha\in \{0,1\}^n} S(\mathbf w^\alpha)\pdv^\alpha f(\mathbf t)\in \I_+[\z,\mathbf t].\] Restrict to $\mathbf{t}=\mathbf{0}$ and use Lemma \ref{lem:restrictionI_+}. We have 
\(S(f)\in \I_+[\z]\).
\end{proof}
 
Theorem~\ref{thm:ma-linear-preserver} gives the multi-affine symbol criterion.
In the next section, we use it to derive the closure properties of the class
\(\I\) and then pass to the general linear preserver theorem. 
\section{Closure Properties  }
We collect the main closure properties of ideal
G{\aa}rding polynomials.
\begin{thm}
\label{thm:positive-affine-pullback-and-directional-derivatives}
The class \(\I\) satisfies the following closure properties.

\begin{enumerate}[label=\textup{(\arabic*)}]
\item \label{enu:closure-polarization}
Let \(\kappa\in\N^n\). If \(f\in\I_\kappa[\x]\), then
\[
\pol_\kappa f\in \I^{\mathfrak a}[\x_\kappa].
\]
The same implication holds with \(\I\) replaced by \(\I_+\).
\item \label{enu:closure-pullback}
Assume that \(\mu:\R^n\to\R^m, \mu(\x)=A\x+\b\) is a positive affine map with $\b\geq0$.
Then for  $f\in \I_{+,m}$, $\mu^*f\in\I_{+,n}$.
If \(\mu\) is strictly positive, then for  
\(
f\in\I_m\), \(\mu^*f\in\I_n.\)

\item \label{enu:closure-specialization}
Let \(\kappa\in\N^n\). If \(F\in\I[\x_\kappa]\), then
\(
\proj_\kappa F\in\I[\x].
\)
The same implication holds with \(\I\) replaced by \(\I_+\).

\item \label{enu:closure-directional}
Let \(\v\in\overline{\Gamma_n^+}\). If  \(f\in\I_{+,n}\),
\(
D_\v f\in\I_{+,n}
\);
if \(f\in\I_n\), then
\(
D_\v f\in\I_n.
\)

\item \label{enu:closure-product}
If \(f,g\in\I_+[\x]\), then \(fg\in\I_+[\x]\); if \(f,g\in\I[\x]\), then
\(fg\in\I[\x]\).
\end{enumerate}
\end{thm}

\begin{proof} \ref{enu:closure-polarization} follows from
Theorem~\ref{thm:polarizingidealgarding}.

%For the rest, we first prove the \(\I_+\)-statements.

For \ref{enu:closure-pullback}, let \(f\in\I_{+,m}\cap\R_\kappa[\y]\) and set
\begin{equation}\label{functionF}
    F:=\pol_\kappa f\in\I_+^{\mathfrak a}[\y_\kappa]
\end{equation}
by \ref{enu:closure-polarization}. Define
\[
\widetilde T_\mu:\R^{\mathfrak a}[\y_\kappa]\longrightarrow \R[\x],
\qquad
\widetilde T_\mu(H):=
H\bigl((\mu_i(\x))_{i\in[m],\,s\in[\kappa_i]}\bigr).
\]
Then
\[
\widetilde T_\mu(F)=f(\mu(\x))=\mu^*f.
\]
Its symbol is
\begin{equation}
\label{eq:affine-pullback-symbol}
\Sym_{\widetilde T_\mu}(\x,\u_\kappa)
=
\prod_{i=1}^m\prod_{s=1}^{\kappa_i}\bigl(\mu_i(\x)+u_{is}\bigr)\in \S_+.
\end{equation} Hence
\(
\Sym_{\widetilde T_\mu}\in \I_+.
\)
By Theorem~\ref{thm:ma-linear-preserver},
\[
\mu^*f=\widetilde T_\mu(F)\in \I_{+,n}.
\]

If \(f\in\I_m\), and \(\mu\) is  strictly
positive, we may choose some large \(\a\in\Gamma_n^+\) with  \(\mu(\a)\in\C_f\). 
Then \(g(\y):=f(\y+\mu(\a))\in\I_{+,m}[\y]\) by Lemma~\ref{lem:translation-invariance}, and
\[
(\mu^*f)(\x+\a)=f(\mu(\x+\a))=g(A\x).
\]
By the \(\I_+\)-case of
\ref{enu:closure-pullback}, \(g(A\x)\in\I_+\). Hence \(\mu^*f\in\I_n\) by
translation invariance.

\ref{enu:closure-specialization} follows from
\ref{enu:closure-pullback} since the $\kappa$ specialization is strictly positive.

  For \ref{enu:closure-directional}, let \(f\in\I_{+,n}\cap\R_\kappa[\x]\) and use \eqref{functionF} again. Define
\[
\widetilde D_\v:\R^{\mathfrak a}[\x_\kappa]\longrightarrow \R[\x],
\qquad
\widetilde D_\v(H):=
D_\v\bigl(\proj_\kappa H\bigr).
\]
Then
\(
\widetilde D_\v(F)=D_\v f,
\)
with
\begin{equation}
\label{eq:directional-symbol}
\Sym_{\widetilde D_\v}(\x,\u_\kappa)
=
D_\v\left(
\prod_{i=1}^n\prod_{s=1}^{\kappa_i}(x_i+u_{is})
\right)\in \S_+.
\end{equation} Therefore
 Theorem~\ref{thm:ma-linear-preserver} gives
\(
D_\v f=\widetilde D_\v(F)\in\I_{+,n}.
\)

For $f\in \I$, choose \(\a\in\C_f\) and set
\[
g(\x):=f(\x+\a)\in\I_{+,n}.
\]
Since \(D_\v g(\x)=D_\v f(\x+\a)\), the \(\I_+\)-case gives
\(D_\v g\in\I_{+,n}\). Hence \(D_\v f\in\I_n\) by translation invariance.

For \ref{enu:closure-product}, we use Lemma~\ref{lem:external-products} to get 
\(
f(\x)g(\y)\in\I[\x,\y].
\)
Let
\(
\Delta:\R^n\to\R^{2n},\) \(
\Delta(\x)=(\x,\x).
\)
Since \(\Delta\) is strictly positive, \ref{enu:closure-pullback} and
\(
fg=\Delta^*H
\) give the corresponding results in \ref{enu:closure-product}.

\end{proof}

We   combine the multi-affine linear preserver theorem with the
polarization--specialization bridge to deduce the general linear preserver
theorem.
\begin{thm}
\label{thm:linear-preserver}
Let \(\kappa\in\N^n\), \(\gamma\in\N^m\), and let
\[
T:\R_\kappa[\x]\longrightarrow \R_\gamma[\y]
\]
be linear. Define its symbol by
\[
\Sym_T(\y,\u):=
T\left(\prod_{i=1}^n (x_i+u_i)^{\kappa_i}\right).
\]
If \(\Sym_T(\y,\u)\in \I_+[\y,\u]\), then
\[
T\bigl(\I_+\cap \R_\kappa[\x]\bigr)\subseteq \I_+\cap \R_\gamma[\y].
\]
\end{thm}

\begin{proof}
Let \(f\in \I_+\cap \R_\kappa[\x]\), and set
\(
F:=\pol_\kappa f\in \I_+^{\mathfrak a}[\x_\kappa]
\)
by Theorem~\ref{thm:positive-affine-pullback-and-directional-derivatives} \ref{enu:closure-polarization}.
Define 
\(\widetilde{T}:\R^{\mathfrak a}[\x_\kappa]\longrightarrow \R[\y_\gamma]\),
\(\widetilde T:=\pol_\gamma\circ T\circ \proj_\kappa.\)
Then
\(\widetilde T(F)=\pol_\gamma(Tf).\) A direct computation shows that
\[
\Sym_{\widetilde T}(\y_\gamma,\u_\kappa)
=
\pol_{\gamma\oplus\kappa}\bigl(\Sym_T(\y,\u)\bigr)
\]
which is in $\I_+[\y_\gamma,\u_\kappa]$ by Theorem~\ref{thm:positive-affine-pullback-and-directional-derivatives} \ref{enu:closure-polarization}. 
Hence Theorem~\ref{thm:ma-linear-preserver} yields
\[
\pol_\gamma(Tf)=\widetilde T(F)\in \I_+[\y_\gamma].
\]
Finally, by Theorem~\ref{thm:positive-affine-pullback-and-directional-derivatives} \ref{enu:closure-specialization},
\(
Tf=\proj_\gamma\bigl(\pol_\gamma(Tf)\bigr)\in \I_+\cap\R_\gamma[\y].
\)
\end{proof}
We next record a basic consequence of the closure properties and the binary
relations introduced earlier.
  \begin{lem}
\label{lem:derivative-position-domination}
Let \(f\in\I[\x]\). Then, for every \(i\in[n]\),
\[
\partial_i f\ip f.
\]
More generally, for every multi-index \(\alpha\), one has
\[
\partial^\alpha f\bvtl f.
\]

\end{lem}

\begin{proof}
We first prove \(\partial_i f\ip f\). Choose
\(m\ge \max\{\deg_{x_i}f,1\}\). By Proposition~\ref{prop:ppol},
\(\ppol_m^{x_i}f\in\I\). Moreover,
\[
f(\x)+z\partial_i f(\x)
=
\ppol_m^{x_i}f(\x,x_i+mz).
\]
The map \((\x,z)\mapsto(\x,x_i+mz)\) is strictly positive affine. Hence,
by Theorem~\ref{thm:positive-affine-pullback-and-directional-derivatives},
\(f(\x)+z\partial_i f(\x)\in\I[\x,z]\). Therefore
\(\partial_i f\ip f\).

Now suppose \(\partial^\alpha f\not\equiv0\). Choose a chain
\(f_0=f,f_1,\ldots,f_N=\partial^\alpha f\), where
\(f_{q+1}=\partial_{j_q}f_q\). The first part gives
\(f_{q+1}\ip f_q\) for each \(q\). By
Remark~\ref{rem:iptoidealdominate}, this implies \(f_{q+1}\bvtl f_q\).

Let \(\beta\) be a multi-index such that
\(\partial^\beta f_N\not\equiv0\). Then
\[
\frac{\partial^\beta f_N}{\partial^\beta f_0}
=
\prod_{q=0}^{N-1}
\frac{\partial^\beta f_{q+1}}{\partial^\beta f_q}.
\]
 By Remark~\ref{rem:iptoidealdominate}, each quotient
\(
\frac{\partial^\beta f_{q+1}}{\partial^\beta f_q}
\)
is log-convex on \(\C_{\partial^\beta f_q}\), hence also on
\(\C_{\partial^\beta f_0}\).
Therefore, their product is log-convex, and hence convex, on
\(\C_{\partial^\beta f_0}\). This is precisely
\(\partial^\alpha f\bvtl f\).
\end{proof}

\begin{example}
    Notice that $\sigma_{k-1}(\x)y+\sigma_k(\x)=\sigma_k(\x,y)$ is real stable. Thus, $\sigma_{k-1}\ip \sigma_k(\x)$ and $\sigma_{k-1}\ip\sigma_k(\x,y)$. For any $l\leq k-2$, we have \[\sigma_l(\x)\ip \sigma_{l+1}(\x)\ip\cdots\ip\sigma_{k-1}(\x)\ip\sigma_k(\x,y).\] A similar argument as in Lemma \ref{lem:derivative-position-domination} shows $\sigma_l(\x)\bvtl \sigma_k(\x,y)$. By Lemma \ref{lem:basic-domination}, for any $c_l\geq0$, \( \sum_{l=0}^{k-2} c_l\sigma_l(\x) \bvtl \sigma_k(\x,y).\) Hence  \[f(\x,y)=\sigma_k(\x)+y\sigma_{k-1}(\x)-\sum_{0\leq l\leq {k-2}} c_l\sigma_l(\x)\in\I[\x,y]\] This example appears in the work of Guan--Zhang on geometric PDEs \cite{MR4278951}.
\end{example}

\section{Concavity, Homogenization and  Lorentzian Polynomials}\label{sec:G=0000E5rding-and-Lorentzian}
We connect ideal G{\aa}rding polynomials with concavity
properties and with Lorentzian polynomials via homogenization. This yields
the analytic and Lorentzian characterizations appearing in the main
structure theorem.
In particular, we prove that \ref{enu:ideal}, \ref{enu:maththmIdealtoconcave}, \ref{enu:maththmIdeallogconcave}, and \ref{enu:maththmIdealtLorentzian} in Theorem \ref{thm:mainstructure} are equivalent.   

The main result of this section is the following.
\begin{thm}
\label{thm:ideal to concavity}Let \(f\in\GS\) have degree \(d\). Then the following are equivalent:  
\begin{enumerate}[label=(\arabic*)]
\item $f\in\I^{d}$;\label{enu:ideal} 
\item \label{enu:Idealtogard}For any non-trivial $\partial^{\alpha}f,$
$(\partial^{\alpha}f){}^{\frac{1}{\deg\partial^{\alpha}f}}$ is concave
in $\C_{\partial^{\alpha}f}$;
\item \label{enu:idealtogardlogconcave}For any non-trivial $\partial^{\alpha}f,$
$\log\partial^{\alpha}f$ is concave in $\C_{\partial^{\alpha}f}$.
\item For any non-trivial $\partial^{\alpha}f$, $\x_0\in \C_{\pdv^\alpha f}$, $\HH_{\x_0} \pdv^\alpha f$ is Lorentzian.\label{enu:idealtoLoren}
\end{enumerate}
\end{thm}
 
\subsection{${\frac 1 d}$-Concavity} We prove the $\frac{1}{d}$-concavity of an ideal \gar{} polynomial in its \gar{} component. 
\begin{prop}\label{prop:ideal to concavity}
    If $f\in\I^d$, then $f^{\frac{1}{d}}$ is concave in $\C_f$.
\end{prop}
\begin{proof}We start with a few reductions. By Theorem~\ref{thm:positive-affine-pullback-and-directional-derivatives}
 \ref{enu:closure-polarization},
\ref{enu:closure-specialization},
it is sufficient to prove the claim for multi-affine ideal Gårding
polynomials. Also, by a proper translation, we may assume that
$f\in\G_{+}$ without loss of generality. Therefore, it is sufficient
to prove the claim for $\Ia_{+},$ multi-affine ideal Gårding polynomials
of non-negative coefficients.

We argue by induction on $d,$ the degree of $f\in\Ia_{+}$. If \(d\le2\), then the degree-two G{\aa}rding and stable classes coincide
\cite[Theorem~4.16]{Fang-MaGard};  $f^{\frac{1}{d}}$ is concave in $\C_{f}$.

For a fixed $d\geq2,$ assume that for any $g\in\I^{m}\cap\G_{+}$ with
$m\leq d$, $g^{\frac{1}{m}}$ is concave in $\C_{g}$. 

Suppose that $f(y,\x)\in\I^{d+1}[y,\x]\cap\G_{+}$. Define $g(\x)=\pdv_{y}f(0,\x)\in\Ia\cap\P^{d}$
and $h(\x)=f(0,\x)$. Then by the multi-affine property of $f,$we
have
\[
f(y,\x)=yg(\x)+h(\x),\ \C_{f}=\{(y,\x)|f(y,\x)>0,\x\in\C_{g}\}.
\]
Since $f\in\I[y,\x]$, the function $\phi(\x):=\frac{h(\x)}{g(\x)}$
is concave for $\x\in\C_{g}$. Hence, $y+\phi(\x)$ is concave and
positive in $\C_{f}$. By the induction hypothesis, $g^{\frac{1}{d}}$
is concave in $\C_{g}$ for $g\in\text{I}^{d}[\x]$. We then use Lemma
\ref{lem:bootstrapconcavity} to conclude that, for $f(y,\x)=g(\x)\cdot (y+\phi(\x))$,  $f^{\frac{1}{d+1}}$
is concave in $\C_{f}$. We  have finished the induction argument.
\end{proof}

\subsection{Lorentzian polynomials}

We recall the definition of Lorentzian polynomials in \cite{BrandenHuh20}. 

\begin{defn}[Lorentzian polynomials]\label{def:Lorentzian_poly} 
    A homogeneous degree $d\geq2$ polynomial $f(\x)=\sum_\alpha c_\alpha\x^\alpha$ with non-negative coefficients is \emph{Lorentzian} if it satisfies 2 conditions:
\begin{enumerate}[label=(\arabic*)]
    \item For all multi-index $\alpha$ with $|\alpha|=d-2$, $\pdv^\alpha f$ is real stable.\label{enu:Lorentzian1}
    \item The support \[\mathrm{supp}(f):=\{\alpha\in\N^n:\,c_\alpha\not=0\}\] is M-convex. \label{enu:LorentzianMconvex} 
\end{enumerate} 
We denote $\L^d$ the class of Lorentzian polynomials of degree $d$.
\end{defn}
A homogeneous polynomial $f$ is said to have M-convex \cite[Chapter 4]{Murota03} support if for any $i\in[n]$ and any $\alpha,\beta\in\text{supp}(f)$
whose $i$-th coordinates satisfy $\alpha_{i}>\beta_{i}$, there is
an index $j$ satisfying $\alpha_{j}<\beta_{j}$, and 
\begin{equation}
\alpha-e_{i}+e_{j}\in\text{supp}(f),\ \beta-e_{j}+e_{i}\in\text{supp}(f).\label{eq:-17}
\end{equation}

Special sub-classes of Lorentzian polynomials include bivariate
polynomials whose coefficients are ultra log-concave  \cite[Example 2.26]{BrandenHuh20}, elementary symmetric
functions, and homogeneous positive real stable polynomials \cite[Proposition 2.2]{BrandenHuh20}.

A remarkable result proved by \cite{BrandenHuh20,ALOVi,ALOVii,ALOViii}
relates Lorentzian and log-concave polynomials.
\begin{thm}
\label{thm:(B-H,AOV,ALOV).-If-} The following
statements hold. 
\begin{enumerate}[label=(\arabic*)]
\item If $f\in\L^d_n$, $f^{\frac{1}{d}}$ is concave in $\Gamma_n^+$. 
\item A homogeneous polynomial $f$ with non-negative coefficients is Lorentzian if and only if for any $\alpha\in\N^{n}$
either 
\[
\pdv^{\alpha}f\equiv0,\ \mathrm{or}\ \log(\pdv^{\alpha}f)\ \mathrm{is\ concave\ in\ }\Gamma_n^+.
\]\label{enu:stronglogconcave}
\end{enumerate}
\end{thm}
 
Homogeneous \gar{} polynomials are closely related to  Lorentzian polynomials. In fact,  for $d\geq2$, the class of degree $d$ homogeneous \gar{} polynomials is sandwiched between real-stable and Lorentzian ones: \[\S_{n}^{d,\mathfrak{h}}\subset\GS_{n}^{d,\mathfrak{h}}\subset\L_{n}^{d}.\] 
When $d=2$, all three classes above coincide. When $d=3$, $\S_{n}^{3,\mathfrak{h}}=\GS_{n}^{d,\mathfrak{h}}$. In other cases, the inclusions are strict. See \cite[section 11]{Fang-MaGard}. 

In general,  homogenization of a \gar{} polynomial with non-negative coefficients
is not Lorentzian. The missing key property   is  log-concavity. On the other hand, for ideal \gar{} polynomials, log-concavity in
\gar{} components is ensured by Proposition \ref{prop:ideal to concavity}.
\begin{thm}
\label{thm:Suppose--is} Suppose that $f\in\I_{n,+}^{d}$. Then $\HH f\in\L_{n+1}^{d}$.
\end{thm}

We first verify the $\M$-convexity condition for the support.

\begin{lem}
Let $f(\mathbf{x})$ be a Rayleigh polynomial in $\mathbb{R}^{n}$ with $f(\mathbf{0})\not=0$.
Then $\HH f(\mathbf{x},y)$ has $\M$-convex support. 
\end{lem}

\begin{proof}
By Lemma 2.22 in \cite{BrandenHuh20}, there is an M-convex set  $J\subset\N^{n+1}$ such that $\mathrm{supp}(f)=\pi_{n+1}(J)$.  $J$ is contained in some
hyperplane \[H=\{\hat{\alpha}:\sum_{i=1}^{n+1}\alpha_{i}=r\in\N\}.\] Note $d=\deg f\leq r$ and  $\alpha_{n+1}\geq r-d$ for all $\hat\alpha\in J$. Then \[J':=J+(d-r)\mathbf{e}_{n+1}\subset\mathbb{N}^{n+1}\] is also $\M$-convex.
Notice that $J'$ is exactly the support of $\HH f.$ We have finished
the proof. 
\end{proof}

\begin{proof}[Proof of Theorem~\ref{thm:Suppose--is}]
We argue by induction on \(d=\deg f\). For \(d=2\), the degree-two
G{\aa}rding and stable classes coincide by
\cite[Theorem~4.16]{Fang-MaGard}. Hence \(f\in\S_n\), and \(\HH f\) is
Lorentzian by \cite[Proposition~2.2]{BrandenHuh20}.

Assume \(d=k\ge3\), and assume the statement has been proved in all
degrees \(<k\). By Theorem~\ref{thm:polarizingidealgarding}, it suffices
to prove the claim after replacing \(f\) by a \(\kappa\)-polarization. The
result then descends to \(f\) by diagonal specialization, since Lorentzian
polynomials are preserved under diagonal specialization. Thus we may assume
\(f\in\G_{n,+}\).

We may also assume \(f(\mathbf 0)>0\). Indeed, for \(\epsilon>0\), the
translate \(f_\epsilon(\x):=f(\x+\epsilon\mathbf 1)\) belongs to \(\I_+\)
and satisfies \(f_\epsilon(\mathbf 0)>0\). If the theorem is proved for
\(f_\epsilon\), then \(\HH f_\epsilon\to\HH f\) coefficientwise as
\(\epsilon\to0^+\), and the closedness of the Lorentzian class gives the
claim for \(f\).

Since \(f\in\GS_+\), Theorem~\ref{thm:GardtoRayleigh} implies that \(f\)
is Rayleigh. Hence, by the preceding support lemma, \(\HH f\) has
\(M\)-convex support. It remains to verify the quadratic Hessian condition:
for every \(\alpha\in\mathbb N^n\) and \(r\in\mathbb N\) with
\(|\alpha|+r=d-2\), the Hessian of
\(\partial_y^r\partial^\alpha \HH f\) has at most one positive eigenvalue
on \(\Gamma_{n+1}^+\).

For each \(i\), the induction hypothesis applied to
\(\partial_i f\in\I_+^{d-1}\) gives
\[
\HH(\partial_i f)(\x,y)
=
(\partial_i f)(\x/y)y^{d-1}
=
\partial_{x_i}\HH f(\x,y)
\in\L_{n+1}^{d-1}.
\]
Therefore all cases with \(|\alpha|\ge1\) follow from the Lorentzian
property of \(\partial_{x_i}\HH f\). Hence it remains only to treat
% \(\alpha=0\) and \(r=d-2\), i.e. test the Hessian condition of $\pdv^{d-2}_y\HH f$  at origin. For such a point, consider the invertible
% transform
% \[
% (\u,y)\longmapsto(\x_0y+\u,y).
% \]
% Under this change, \(\HH f\) becomes the localized homogenization
% \[
% \HH_{\x_0}f(\u,y)
% =
% f(\x_0+\u/y)y^d.
% \]
% Hessian signatures are preserved under this change by congruence. Thus it
% is enough to compute the Hessian of
% \(\partial_y^{d-2}\HH_{\x_0}f\) at \((\mathbf 0,1)\).
A direct Taylor expansion at the origin gives
\[
\operatorname{Hess}
\bigl(\partial_y^{d-2}\HH f\bigr)(\mathbf 0,0)
=
(d-2)!
\left[
\begin{array}{cc}
\operatorname{Hess}f & (d-1)\nabla f\\
(d-1)(\nabla f)^T & d(d-1)f
\end{array}
\right]_{\x=\mathbf 0}.
\]
Since \(\mathbf 0\in \C_f\), one has \(f(\mathbf 0)>0\). Hence this
matrix is congruent to
\[
(d-2)!
\left[
\begin{array}{cc}
\operatorname{Hess}f-\frac{d-1}{d}\frac{(\nabla f)(\nabla f)^T}{f} & 0\\
0 & d(d-1)f
\end{array}
\right]_{\x=\mathbf 0}.
\]
By Proposition~\ref{prop:ideal to concavity}, \(f^{1/d}\) is concave on
\(\C_f\). Therefore
\[
\left[
\mathrm{Hess} f
-
\left(1-\frac1d\right)
\frac{(\nabla f) (\nabla f)^T}{f}
\right]_{\x=\mathbf 0}
\]
as a matrix is negative semidefinite. Thus the above congruent matrix has at most one
positive eigenvalue, coming from the \(d(d-1)f(\mathbf 0)\) entry. Hence the
same is true for the original Hessian of
\(\partial_y^{d-2}\HH f\) at \((\mathbf 0,0)\).

Therefore \(\HH f\) satisfies the defining Hessian condition for
Lorentzian polynomials, and since its support is \(M\)-convex, we conclude
that \(\HH f\in\L_{n+1}^d\).
\end{proof}
Finally, we conclude the proof of Theorem \ref{thm:ideal to concavity}.
\begin{proof}[Proof of Theorem \ref{thm:ideal to concavity}]
We first prove the equivalence of \ref{enu:ideal}, \ref{enu:Idealtogard}, and \ref{enu:idealtogardlogconcave}.
    \ref{enu:ideal} $\Rightarrow$\ref{enu:Idealtogard} is due to Proposition \ref{prop:ideal to concavity}.

   \ref{enu:Idealtogard} $\Rightarrow$ \ref{enu:idealtogardlogconcave} is trivial since any $\frac{1}{d}$-concave positive function is automatically log-concave.
    
    \ref{enu:idealtogardlogconcave} implies the convexity of $\C_{\pdv^\alpha f}$ and hence implies \ref{enu:ideal}. 

    \ref{enu:ideal} $\Rightarrow$ \ref{enu:idealtoLoren} is due to Theorem \ref{thm:Suppose--is}.

    \ref{enu:idealtoLoren}  $\Rightarrow$ \ref{enu:idealtogardlogconcave} is due to Theorem \ref{thm:(B-H,AOV,ALOV).-If-} \ref{enu:stronglogconcave}. Since \ref{enu:idealtogardlogconcave} implies \ref{enu:ideal}, we have finished the proof.
\end{proof}
We have thus completed the proof of the equivalence between ideality,
quotient-type concavity, logarithmic concavity, and localized Lorentzian
homogenization. Combined with
Theorem~\ref{thm:polarizingidealgarding} and
Theorem~\ref{thm:positive-affine-pullback-and-directional-derivatives},
this yields Theorem~\ref{thm:mainstructure}. We have also completed the proof of Corollary~\ref{cor:InversionofBH}.

\section{Newton-Maclaurin   Inequalities}
We record several Newton--Maclaurin-type inequalities and quotient
concavity consequences of the preceding theory. 

\subsection{Directional specializations and root sequences}

We assume \(\v\ge \mathbf 0\) and \(\v\neq \mathbf 0\).

For \(d\ge1\), let \(\MRS_d\) be defined as in
Definition~\ref{MSD_d}. We embed \(\MRS_d\) into
\((\R\cup\{-\infty\})^{\mathbb N}\) by
\[
(r_0,\ldots,r_{d-1})
\longmapsto
(r_0,\ldots,r_{d-1},-\infty,-\infty,\ldots).
\]
Set
\[
\MRS:=\bigcup_{d\ge1}\MRS_d
\subset (\R\cup\{-\infty\})^{\mathbb N},
\]
with the induced product topology. Thus, lower-degree root sequences are
realized as boundary limits of higher-degree ones by sending the missing
roots to \(-\infty\).

Let \(f\in\GS[\x]\), \(\x\in\C_f\), and
\(\v\ge\mathbf 0\), \(\v\neq\mathbf 0\). Define
\begin{equation}
  \label{directional specialization}  
F_{\x,\v}(t):=f(\x+t\v).
\end{equation}

If \(F_{\x,\v}\not\equiv0\), then \(F_{\x,\v}\) is a univariate
G{\aa}rding polynomial. Let
\(
d(\x,\v):=\deg F_{\x,\v}.
\)
For \(0\le k\le d(\x,\v)-1\), set
\(
r_k(\x,\v):=r\bigl(F_{\x,\v}^{(k)}\bigr).
\)
The directional specialization defines a natural map
\[\begin{aligned}
    R_f&:\C_f\times
\bigl(\overline{\Gamma_n^+}\setminus\{\mathbf 0\}\bigr)
\longrightarrow \MRS,\\
(\x,\v)
&\mapsto
(r_0(\x,\v),\ldots,r_{d(\x,\v)-1}(\x,\v),-\infty,\ldots).
\end{aligned}
\]

The two Newton--Maclaurin type results below are obtained from this same
specialization. The first uses the univariate ultra log-concavity property of
\(F_{\x_0,\v}\), while the second uses the stronger hypothesis
\(f\in\I[\x]\), together with derivative closure and quotient concavity.

 \begin{thm}
\label{garding-NM}
Let \(f\in\GS[\x]\), let
\(\v\in \overline{\Gamma_n^+}\setminus\{\mathbf 0\}\), and let
\(\x_0\in\C_f\). Suppose that \(\ell\ge1\) satisfies
\[
(D_\v^k f)(\x_0)>0 \quad (0\le k\le \ell),
\qquad
(D_\v^{\ell+1}f)(\x_0)=0.
\]
Then, for \(1\le j\le \ell-1\),
\[
\bigl((D_\v^j f)(\x_0)\bigr)^2
\ge
\frac{\ell-j+1}{\ell-j}\,
(D_\v^{j-1}f)(\x_0)\,(D_\v^{j+1}f)(\x_0).
\]
\end{thm}

\begin{proof}
By discussion above, $F_{\x_0,\v}(t)$ is univariate G\aa rding and its Taylor expansion at $t=0$ is
\[
 \sum_{k\ge0}\frac{(D_\v^k f)(\x_0)}{k!}t^k.
\]
By Theorem 15.2 in \cite{Fang-MaGard}, its coefficient sequence
\(
\left\{
(D_\v^k f)(\x_0)/k!
\right\}_{k=0}^{\ell}
\)
is ultra log-concave (ULC), which means that this sequence has no internal zero and for \(1\le j\le \ell-1\),
\[
\left(
\frac{(D_\v^j f)(\x_0)/j!}{\tbinom{\ell}{j}}
\right)^2
\ge
\left(
\frac{(D_\v^{j-1} f)(\x_0)/(j-1)!}{\tbinom{\ell}{j-1}}
\right)
\left(
\frac{(D_\v^{j+1} f)(\x_0)/(j+1)!}{\tbinom{\ell}{j+1}}
\right).
\]
Simplifying the factorial and binomial factors gives the claimed inequality.
\end{proof}

\begin{rem}
\label{rem:sigma-k-NM}
When \(f=\sigma_k(\x)\) and \(\v=\mathbf 1=(1,\ldots,1)\), Theorem
\ref{garding-NM} gives for $r\leq k-1$,
\[
\left(\frac{\sigma_r(\x)}{\tbinom{n}{r}}\right)^2
\ge
\frac{\sigma_{r-1}(\x)}{\tbinom{n}{r-1}}
\cdot
\frac{\sigma_{r+1}(\x)}{\tbinom{n}{r+1}}.
\]
which is the usual Newton--Maclaurin inequality.  Thus, the theorem extends
the classical coefficient-level Newton--Maclaurin inequality from
\(\sigma_k\) to arbitrary \(\gar{}\) polynomials and arbitrary nonnegative
directions.
\end{rem}
\begin{thm}
\label{thm:newton-maclaurin-quotient}
Let \(f\in\I[\x]\).
\begin{enumerate}[label=\textup{(\arabic*)}]
\item\label{enu:one-step-directional}
Let \(\v\ge \mathbf 0\), \(\v\neq \mathbf 0\), and assume
\(D_\v f\not\equiv0\). Then \(D_\v f\in\I[\x]\), and \(f/D_\v f\) is
concave on \(\C_{D_\v f}\).

\item\label{enu:higher-directional}
Let \(\v_1,\ldots,\v_m\ge \mathbf 0\), with \(\v_i\neq\mathbf 0\), and
assume \(D_{\v_m}\cdots D_{\v_1}f\not\equiv0\). Then
\[
\left(
\frac{f}{D_{\v_m}\cdots D_{\v_1}f}
\right)^{1/m}
\]
is concave on \(\C_f\).
\end{enumerate}
\end{thm}

\begin{proof}
We prove \ref{enu:one-step-directional}. Set
\(F(\u,y):=f(\u+y\v)\). By Theorem~\ref{thm:positive-affine-pullback-and-directional-derivatives} \ref{enu:closure-pullback} and Definition~\ref{def:ideal}, \(F\in\I[\u,y]\) and
\(D_y F\in\I[\x,y]\). By Theorem~\ref{prop:fibre-quotient-concavity},
\(F/\partial_yF\) is concave on \(\C_{\partial_yF}\). Restricting to the
slice \(y=0\) gives
\[
\frac{F(\u,0)}{\partial_yF(\u,0)}
=
\frac{f(\u)}{D_\v f(\u)},
\]
so \(f/D_\v f\) is concave on \(\C_{D_\v f}\).

Part \ref{enu:higher-directional} follows by iterating
\ref{enu:one-step-directional} and applying
Lemma~\ref{lem:bootstrapconcavity}.
\end{proof}

\begin{rem}
The one-step quotient naturally lives on the larger denominator
component $\C_{D_{\v}f}$. For higher quotients, positivity of all
intermediate factors is required, so the natural common domain is
$\C_f$, or more generally any convex domain on which all successive
quotients remain positive.
\end{rem}

\begin{rem}
For $f=\sigma_k$ and $\v=\mathbf 1$, part
\ref{enu:one-step-directional} recovers the classical concavity of $
\frac{\sigma_k}{\sigma_{k-1}}$
on $\Gamma^+_{k-1}$, while part \ref{enu:higher-directional} yields
the quotient concavity
\(
\left(
\frac{\sigma_k}{\sigma_{k-m}}
\right)^{1/m}
\)
on $\Gamma^+_k$.
\end{rem}

\printbibliography

\end{document}